\documentclass[wcp]{jmlr}

\usepackage{longtable}

\usepackage{amsmath,amssymb,amsfonts}
\usepackage{graphicx}
\usepackage{graphics}
\usepackage{textcomp}
\usepackage{xcolor}
\usepackage{color}

\usepackage{multirow}
    
\def\b{\ensuremath\boldsymbol}

\usepackage{adjustbox}

\usepackage{booktabs}

\makeatletter
\let\Ginclude@graphics\@org@Ginclude@graphics 
\makeatother

\jmlrvolume{}
\jmlryear{2021}
\jmlrworkshop{ACML 2021}

\title[Vector Transport Free Riemannian LBFGS]{Vector Transport Free Riemannian LBFGS for Optimization on Symmetric Positive Definite Matrix Manifolds}

  \author{\Name{Reza Godaz}\thanks{The first two authors contributed equally to this work.} \Email{reza.godaz@mail.um.ac.ir}\\
  \addr Department of Computer Engineering, Ferdowsi University of Mashhad, Mashhad, Iran
  \AND
  \Name{Benyamin Ghojogh}\footnotemark[1] \Email{bghojogh@uwaterloo.ca}\\
  \addr Department of Electrical and Computer Engineering, University of Waterloo, ON, Canada
  \AND
  \Name{Reshad Hosseini} \Email{reshad.hosseini@ut.ac.ir}\\
  \addr Department of Electrical and Computer Engineering, University of Tehran, Tehran, Iran
  \AND
  \Name{Reza Monsefi} \Email{monsefi@um.ac.ir}\\
  \addr Department of Computer Engineering, Ferdowsi University of Mashhad, Mashhad, Iran
  \AND
  \Name{Fakhri Karray} \Email{karray@uwaterloo.ca}\\
  \Name{Mark Crowley} \Email{mcrowley@uwaterloo.ca}\\
  \addr Department of Electrical and Computer Engineering, University of Waterloo, ON, Canada
 }

\editors{Vineeth N Balasubramanian and Ivor Tsang}

\begin{document}

\maketitle

\begin{abstract}
This work concentrates on optimization on Riemannian manifolds. The Limited-memory Broyden-Fletcher-Goldfarb-Shanno (LBFGS) algorithm is a commonly used quasi-Newton method for numerical optimization in Euclidean spaces. Riemannian LBFGS (RLBFGS) is an extension of this method to Riemannian manifolds. RLBFGS involves computationally expensive vector transports as well as unfolding recursions using adjoint vector transports. In this article, we propose two mappings in the tangent space using the inverse second root and Cholesky decomposition. These mappings make both vector transport and adjoint vector transport identity and therefore isometric. Identity vector transport makes RLBFGS less computationally expensive and its isometry is also very useful in convergence analysis of RLBFGS. Moreover, under the proposed mappings, the Riemannian metric reduces to Euclidean inner product, which is much less computationally expensive. We focus on the Symmetric Positive Definite (SPD) manifolds which are beneficial in various fields such as data science and statistics. This work opens a research opportunity for extension of the proposed mappings to other well-known manifolds. 
\end{abstract}

\section{Introduction}

Various numerical optimization methods have appeared, for the Euclidean spaces, which can be categorized into first-order and second-order methods \citep{nocedal2006numerical}. Examples for the former category are steepest descent and gradient descent and for the latter group are Newton's method. Computation of the Hessian matrix is usually expensive in Newton's method encouraging many practical problems to either use quasi-Newton's methods for approximating the Hessian matrix or use non-linear conjugate gradient \citep{hestenes1952methods}. The most well-known algorithm for quasi-Newton optimization is Broyden-Fletcher-Goldfarb-Shanno (BFGS) \citep{fletcher2013practical}. Limited-memory BFGS (LBFGS) is a simplified version of BFGS which utilizes less memory \citep{nocedal1980updating,liu1989limited}. It has recursive unfoldings which approximate the descent directions in optimization. 

Unlike Euclidean spaces in which the optimization direction lie in a linear coordinate system, Riemannian spaces have curvature in coordinates. 
Recently, extension of optimization methods from Euclidean spaces to Riemannian manifolds has been extensively noticed in the literature \citep{absil2009optimization,boumal2020introduction,hu2020brief}.
For example, Euclidean BFGS has been extended to Riemannian manifolds, named Riemannian BFGS (RBFGS), \citep{qi2010riemannian}, its convergence has been proven \citep{ring2012optimization,huang2015broyden}, and its properties have been analyzed in the literature \citep{seibert2013properties}. 
As vector transport is computationally expensive in RBFGS, cautious RBFGS was proposed \citep{huang2016riemannian} which ignores the curvature condition in the Wolfe conditions \citep{wolfe1969convergence} and only checks the Armijo condition \citep{armijo1966minimization}. 
Since the curvature condition guarantees that the approximation of Hessian remains positive definite, it compensates by checking a cautious condition \citep{li2001global} before updating the approximation of Hessian. This cautious RBFGS has been used in the Manopt optimization toolbox \citep{boumal2014manopt}.
Another approach is an extension of the Euclidean LBFGS to Riemannian manifolds, named Riemannian LBFGS (RLBFGS), using both Wolfe conditions \citep{wolfe1969convergence} in linesearch can be found in \citep{sra2015conic,sra2016geometric,hosseini2020alternative}. 
Some other direct extensions of Euclidean BFGS to Riemannian spaces exist (e.g., see \cite[Chapter 7]{ji2007optimization}).

In this paper, we address the computationally expensive parts of RLBFGS algorithm, which are computation of vector transports, their adjoints, and Riemannian metrics. To achieve this, we propose two mappings, in the tangent space, which make RLBFGS free of vector transport. We name the obtained algorithm Vector Transport Free (VTF)-RLBFGS. One mapping uses inverse second root and the other uses Cholesky decomposition which is very efficient \citep{golub2013matrix}. The proposed mappings make both vector transport and adjoint vector transport, which are used in RLBFGS, identity. This reduction of transports to identity makes optimization much less expensive computationally. Moreover, as the vector transports become identity, they are isometric which is a suitable property mostly used in the convergence proofs of RBFGS and RLBFGS algorithms \citep{ring2012optimization,huang2015broyden}. 
Furthermore, under the proposed mappings, the Riemannian metric reduces to Euclidean inner product which is much less computationally expensive.
In this paper, we concentrate on the Symmetric Positive Definite (SPD) manifolds \citep{sra2015conic,sra2016geometric,bhatia2009positive} which are very useful in data science and machine learning, such as in mixture models \citep{hosseini2020alternative,hosseini2015mixest}. This paper opens a new research path for extension of the proposed mappings to other well-known manifolds such as Grassmann and Stiefel \citep{edelman1998geometry}. 

The remainder of this paper is organized as follows. Section \ref{section_background} reviews the notations and technical background on Euclidean LBFGS, Wolfe conditions, Riemannian LBFGS, and SPD manifold. The proposed mappings using inverse second root and Cholesky decomposition are introduced in Sections \ref{section_VTF_inverse_root} and \ref{section_VTF_Cholesky}, respectively. Simulation results are reported in Section \ref{section_simulations}. Finally, Section \ref{section_conclusion} concludes the paper and proposes the possible future directions.



















\section{Background and Notations}\label{section_background}

\subsection{Euclidean BFGS and LBFGS}\label{section_Euclidean_BFGS}

Consider minimization of the cost function $f(\b{\Sigma})$ where the point $\b{\Sigma}$ belongs to some domain.
In Newton's method, the descent direction $\b{p}_k$ at the iteration $k$ is calculated as $\b{B}_k \b{p}_k = - \nabla f(\b{\Sigma}_k) \implies \b{p}_k = - \b{B}_k^{-1} \nabla f(\b{\Sigma}_k)$ \citep{nocedal2006numerical}, where $\b{B}_k$ is the Hessian or approximation of Hessian and $\nabla f(\b{\Sigma}_k)$ is the gradient of function at iteration $k$.
The Euclidean BFGS method is a quasi-Newton's method which approximates the Hessian matrix as $\b{B}_{k+1} := \b{B}_k + (\b{y}_k \b{y}_k^\top)/(\b{y}_k^\top \b{s}_k) - (\b{B}_k \b{s}_k \b{s}_k^\top \b{B}_k^\top)/(\b{s}_k^\top \b{B}_k \b{s}_k)$ \citep{fletcher2013practical,nocedal2006numerical}, where $\b{s}_k := \b{\Sigma}_{k+1} - \b{\Sigma}_k$ and $\b{y}_k := \nabla f(\b{\Sigma}_{k+1}) - \nabla f(\b{\Sigma}_{k})$. The descent direction is $\b{p}_k$ whose expression was provided above. 


The Euclidean LBFGS calculates the descent direction recursively where it uses the approximation of the inverse Hessian as $\b{H}_{k+1} := \b{V}_k^\top \b{H}_k \b{V}_k + \rho_k \b{s}_k \b{s}_k^\top$ \citep{nocedal1980updating,liu1989limited},
where $\b{H}_k$ denotes the approximation of the inverse of the Hessian at iteration $k$, $\rho_k := 1/(\b{y}_k^\top \b{s}_k)$, and $\b{V}_k := \b{I} - \rho_k \b{y}_k \b{s}_k^\top$ in which $\b{I}$ denotes the identity matrix. 
The LBFGS algorithm updates the approximation of the inverse of the Hessian matrix recursively and for that it always stores a memory window of pairs $\{\b{y}_k, \b{s}_k\}$ \citep{liu1989limited}. 

\subsection{Linesearch and Wolfe Conditions}\label{section_Wolfe_conditions}

After finding the descent direction $\b{p}_k$ at each iteration $k$ of optimization, one needs to know what step size $\alpha_k$ should be taken in that direction. Linesearch should be performed to find the largest step, for faster progress, satisfying Wolfe conditions which are $f(\b{\Sigma}_k + \alpha_k) \leq f(\b{\Sigma}_k) + c_1 \alpha_k \b{p}_k^\top \nabla f(\b{\Sigma}_k)$ and $-\b{p}_k^\top \nabla f(\b{\Sigma}_k + \alpha_k \b{p}_k) \leq -c_2 \b{p}_k^\top \nabla f(\b{\Sigma}_k)$ \citep{wolfe1969convergence}, where the parameters $0 < c_1 < c_2 < 1$ are recommended to be $c_1 = 10^{-1}$ and $c_2 = 0.9$ \citep{nocedal2006numerical}.
The former condition is the Armijo condition to check if cost decreases sufficiently \citep{armijo1966minimization} while the latter is the curvature condition making sure that the slope is reduced sufficiently in a way that the approximation of Hessian remains positive definite. Note that there also exists a strong curvature condition, i.e., $|\b{p}_k^\top \nabla f(\b{\Sigma}_k + \alpha_k \b{p}_k)| \leq c_2 |\b{p}_k^\top \nabla f(\b{\Sigma}_k)|$. 

\subsection{Riemannian Notations}

Consider a Riemannian manifold denoted by $\mathcal{M}$. 
At every point $\b{\Sigma} \in \mathcal{M}$, there is a tangent space to the manifold, denoted by $T_{\b{\Sigma}}\mathcal{M}$. A tangent space includes tangent vectors. We denote a tangent vector by $\b{\xi}$. 
For $\b{\xi}, \b{\eta} \in T_{\b{\Sigma}}\mathcal{M}$, a metric on this manifold is the inner product defined on manifold and is denoted by $g_{\b{\Sigma}}(\b{\xi}, \b{\eta})$. 
Note that the gradient of a cost function, whose domain is a manifold, is a tangent vector in the tangent space and is denoted by $\nabla f(\b{\Sigma})$ for point $\b{\Sigma} \in \mathcal{M}$.
Vector transport is an operator which maps a tangent vector $\b{\xi} \in T_{\b{\Sigma}_1}\mathcal{M}$ from its tangent space at point $\b{\Sigma}_1$ to another tangent space at another point $\b{\Sigma}_2$. We denote this vector transport by $\mathcal{T}_{\b{\Sigma}_1, \b{\Sigma}_2}(\b{\xi}): T_{\b{\Sigma}_1}\mathcal{M} \mapsto T_{\b{\Sigma}_2}\mathcal{M}$.
Now, consider a point $\b{\Sigma}$ on a manifold $\mathcal{M}$ and a descent direction $\b{\xi}$ in the tangent space $T_{\b{\Sigma}}\mathcal{M}$. The retraction $\text{Ret}_{\b{\Sigma}}(\b{\xi}): T_{\b{\Sigma}}\mathcal{M} \mapsto \mathcal{M}$ retracts or maps the direction $\b{\xi}$ in the tangent space onto the manifold $\mathcal{M}$. The operator exponential map, denoted by $\text{Exp}_{\b{\Sigma}}(\b{\xi})$, is also capable of this mapping by moving along the geodesic. 
One can use the second-order Taylor expansion of exponential map, which is positive-preserving \citep{jeuris2012survey} for the case of SPD manifold, for approximating the exponential map. 
In this paper, $\textbf{tr}(.)$ denotes the trace of matrix and $\|.\|_F$ denotes the Frobenius norm.

\subsection{Riemannian BFGS and LBFGS}

The Riemannian extension of Euclidean BFGS (RBFGS) \citep{qi2010riemannian,ring2012optimization,huang2015broyden} performs updates of Hessian approximation by the $\b{B}_{k+1}$ (see Section \ref{section_Euclidean_BFGS}) using Riemannian operators. RBFGS methods check both Wolfe conditions (see Section \ref{section_Wolfe_conditions}) which are computationally expensive. Cautious RBFGS \citep{huang2016riemannian} ignores the curvature condition and only checks the Armijo condition for linesearch. However, for ensuring that the Hessian approximation remains positive definite, it checks a cautious condition \citep{li2001global} before updating the Hessian approximation. 


Riemannian LBFGS \citep{sra2015conic,sra2016geometric,hosseini2020alternative} performs the recursions of LBFGS in the Riemannian space for finding the descent direction $\b{\xi}_k \in T_{\b{\Sigma}_k} {\mathcal{M}}$ and checks both Wolfe conditions in linesearch. 
In every iteration $k$ of optimization, recursion starts with the direction $\b{p} = -\nabla f(\b{\Sigma}_k)$. 
Let the recursive function $\text{GetDirection}(\b{p}, k)$ returns the descent direction.
Inside every step of this recursion, we have \cite[Algorithm 3]{hosseini2020alternative}:
\begin{align}
&\b{\tilde{p}} := \b{p} - \rho_k\, \b{g}_{\b{\Sigma}_k}(\b{s}_k, \b{p}) \b{y}_k, \label{eq_recursion_p_tilde} \\
&\b{\widehat{p}} := \mathcal{T}_{\b{\Sigma}_{k-1}, \b{\Sigma}_k}\big(\text{GetDirection}(\mathcal{T}^*_{\b{\Sigma}_{k-1}, \b{\Sigma}_k} (\b{\tilde{p}}), k-1)\big), \label{eq_recursion_adjoint} \\
&\text{return  } \b{\xi}_k := \b{\widehat{p}} - \rho_k\, \b{g}_{\b{\Sigma}_k}(\b{y}_k, \b{\widehat{p}}_k) \b{s}_k + \rho_k\, \b{g}_{\b{\Sigma}_k}(\b{s}_k, \b{s}_k) \b{p}, \label{eq_recursion_xi}
\end{align}
where, $\rho_k := 1/\b{g}_{\b{\Sigma}_k}(\b{y}_k, \b{s}_k)$, and inspired by the introduced $\b{s}_k$ and $\b{y}_k$ for Euclidean spaces, we have:
\begin{align}
& \b{\Sigma}_{k+1} := \text{Exp}_{\b{\Sigma}_k}(\alpha_k \b{\xi}_k) \,\, \text{or} \,\, \text{Ret}_{\b{\Sigma}_k}(\alpha_k \b{\xi}_k), \label{eq_RLBFGS_Sigma} \\
& \b{s}_{k+1} := \mathcal{T}_{\b{\Sigma}_k, \b{\Sigma}_{k+1}}(\alpha_k \b{\xi}_k), \label{eq_RLBFGS_s} \\
& \b{y}_{k+1} := \nabla f(\b{\Sigma}_{k+1}) - \mathcal{T}_{\b{\Sigma}_k, \b{\Sigma}_{k+1}}\big(\nabla f(\b{\Sigma}_{k})\big). \label{eq_RLBFGS_y}
\end{align}
Note that the new point in every iteration is found by retraction, or an exponential map, of the searched point along the descent direction onto manifold. 
According to Eq. (\ref{eq_recursion_adjoint}) in recursion and Eq. (\ref{eq_RLBFGS_s}), RLBFGS involves both adjoint vector transport and vector transport which are computationally expensive. Moreover, Eqs. (\ref{eq_recursion_p_tilde}) and (\ref{eq_recursion_xi}) show that Riemannian metric is utilized many times inside recursions. Our proposed mappings simplify all vector transport, adjoint vector transport, and metric which are used in the RLBFGS algorithm.

\subsection{Symmetric Positive Definite (SPD) Manifold}

Consider the SPD manifold \citep{sra2015conic,sra2016geometric} whose every point is a SPD matrix, i.e., $\b{\Sigma} \in \mathcal{M}$ and $\mathbb{S}_{++}^{n} \ni \b{\Sigma} \succ \b{0}$. It can be shown that the tangent space to the SPD manifold is the space of symmetric matrices, i.e., $T_{\mathcal{M}}(\b{\Sigma}) \subset \mathbb{S}_{++}^n$ \citep{bhatia2009positive}. 
In this paper, we focus on SPD manifolds which are widely used in data science. 
The operators for metric, gradient, exponential map, and vector transport on SPD manifolds are listed in Table \ref{table_operators} \citep{sra2016geometric,hosseini2020alternative}.
In this table, $\nabla_E f(\b{\Sigma})$ denotes the Euclidean gradient and $\b{L}_1$ and $\b{L}_2$ are the lower-triangular matrices in Cholesky decomposition of points $\b{\Sigma}_1$ and $\b{\Sigma}_2$, respectively. 

\begin{table*}[!t]
\caption{Operators on SPD manifold under the proposed mappings.}
\label{table_operators}
\renewcommand{\arraystretch}{1.3}  
\centering
\scalebox{1}{    
\renewcommand{\arraystretch}{1.2}
\begin{tabular}{l || c}
\hline
\hline
Operator & No mapping \\
\hline
Metric, $g_{\b{\Sigma}}(\b{\xi}, \b{\eta})$ & $ \textbf{tr}(\b{\Sigma}^{-1} \b{\xi} \b{\Sigma}^{-1} \b{\eta})$ \\
Gradient, $\nabla f(\b{\Sigma})$ & $\frac{1}{2} \b{\Sigma} \big( \nabla_E f(\b{\Sigma}) + (\nabla_E f(\b{\Sigma}))^\top \big) \b{\Sigma}$ \\
Exponential map, $\text{Exp}_{\b{\Sigma}}(\b{\xi})$ & $\b{\Sigma}\, \text{exp}(\b{\Sigma}^{-1} \b{\xi}) = \b{\Sigma}^{\frac{1}{2}}\, \text{exp}(\b{\Sigma}^{-\frac{1}{2}} \b{\xi} \b{\Sigma}^{-\frac{1}{2}})\, \b{\Sigma}^{\frac{1}{2}}$ \\
Vector transport, $\mathcal{T}_{\b{\Sigma}_1,\b{\Sigma}_2}(\b{\xi})$ & $\b{\Sigma}_2^{\frac{1}{2}} \b{\Sigma}_1^{-\frac{1}{2}} \b{\xi} \b{\Sigma}_1^{-\frac{1}{2}} \b{\Sigma}_2^{\frac{1}{2}}$ or $\b{L}_2 \b{L}_1^{-1} \b{\xi} \b{L}_1^{-\top} \b{L}_2^{\top}$ \\
Approx. Euclidean retraction, $\text{Ret}_{\b{\Sigma}}(\b{\xi})$ & $\b{\Sigma} + \b{\xi} + \frac{1}{2} \b{\xi} \b{\Sigma}^{-1} \b{\xi}$ \\
\hline
\hline
\hline
Operator & Mapping by inverse second root \\
\hline
Mapping & $\b{\xi}' := \b{\Sigma}^{-\frac{1}{2}}\, \b{\xi}\, \b{\Sigma}^{-\frac{1}{2}}$ \\
Metric, $g'_{\b{\Sigma}}(\b{\xi}', \b{\eta}')$ & $\textbf{tr}(\b{\xi}' \b{\eta}')$ \\
Gradient, $\nabla' f(\b{\Sigma})$ & $\frac{1}{2} \b{\Sigma}^{\frac{1}{2}} \big( \nabla_E f(\b{\Sigma}) + (\nabla_E f(\b{\Sigma}))^\top \big) \b{\Sigma}^{\frac{1}{2}}$ \\
Exponential map, $\text{Exp}_{\b{\Sigma}}(\b{\xi}')$ & $\b{\Sigma}^{\frac{1}{2}}\, \text{exp}(\b{\xi}')\, \b{\Sigma}^{\frac{1}{2}}$ \\
Vector transport, $\mathcal{T}'_{\b{\Sigma}_1,\b{\Sigma}_2}(\b{\xi}')$ & $\b{\xi}'$ \\
Approx. Euclidean retraction, $\text{Ret}_{\b{\Sigma}}(\b{\xi}')$ & $\b{\Sigma} + \b{\Sigma}^{\frac{1}{2}} \b{\xi}' \b{\Sigma}^{\frac{1}{2}} + \frac{1}{2} \b{\Sigma}^{\frac{1}{2}} \b{\xi}'^{\,2} \b{\Sigma}^{\frac{1}{2}}$ \\
\hline
\hline
\hline
Operator & Mapping by Cholesky decomposition \\
\hline
Mapping & $\b{\xi}' := \b{L}^{-1}\, \b{\xi}\, \b{L}^{-\top}$ \\
Metric, $g'_{\b{\Sigma}}(\b{\xi}', \b{\eta}')$ & $\textbf{tr}(\b{\xi}' \b{\eta}')$ \\
Gradient, $\nabla' f(\b{\Sigma})$ & $\frac{1}{2} \b{L}^{\top} \big( \nabla_E f(\b{\Sigma}) + (\nabla_E f(\b{\Sigma}))^\top \big) \b{L}$ \\
Exponential map, $\text{Exp}_{\b{\Sigma}}(\b{\xi}')$ & $\b{\Sigma}\, \text{exp}(\b{L}^{-\top} \b{\xi}' \b{L}^\top)$ \\
Vector transport, $\mathcal{T}'_{\b{\Sigma}_1,\b{\Sigma}_2}(\b{\xi}')$ & $\b{\xi}'$ \\
Approx. Euclidean retraction, $\text{Ret}_{\b{\Sigma}}(\b{\xi}')$ & $\b{\Sigma} + \b{L}\, \b{\xi}' \b{L}^{\top} + \frac{1}{2} \b{L}\, \b{\xi}'^{\,2} \b{L}^{\top}$ \\
\hline
\hline
\end{tabular}%
}
\end{table*}

\section{Vector Transport Free Riemannian LBFGS Using Mapping by Inverse Second Root}\label{section_VTF_inverse_root}

We propose two mappings on tangent vectors in the tangent space where the first mapping is by inverse second root. Our first proposed mapping in the tangent space of every point $\b{\Sigma} \in \mathcal{M}$ is:
\begin{align}\label{eq_mapping}
\b{\xi}' := \b{\Sigma}^{-\frac{1}{2}}\, \b{\xi}\, \b{\Sigma}^{-\frac{1}{2}} \implies \b{\xi} = \b{\Sigma}^{\frac{1}{2}}\, \b{\xi}'\, \b{\Sigma}^{\frac{1}{2}},
\end{align}
where the mapped tangent vector still remains in the tangent space, i.e. $\b{\xi}, \b{\xi}' \in T_{\b{\Sigma}} {\mathcal{M}} \subset \mathbb{S}_{++}^n$. It is important the proposed mapping is bijective and keeps the tangent vector in the tangent space while it simplifies vector transport, adjoint vector transport, and metric.

Under the proposed mapping (\ref{eq_mapping}), the Riemannian operators on a SPD manifold are modified and mostly simplified. These operators are listed in Table \ref{table_operators}. In the following, we provide proofs for these modifications.

\begin{proposition}\label{proposition_mapping_inverseSecondRoot}
After mapping (\ref{eq_mapping}), we have
\begin{itemize}
\item Metric: the metric on SPD manifold is reduced to the Euclidean inner product, i.e., $g_{\b{\Sigma}}(\b{\xi}', \b{\eta}') = \textbf{tr}(\b{\xi}' \b{\eta}')$.
\item Gradient: the gradient on a SPD manifold is changed to  
$\nabla' f(\b{\Sigma}) = \frac{1}{2} \b{\Sigma}^{\frac{1}{2}} \big( \nabla_E f(\b{\Sigma}) + (\nabla_E f(\b{\Sigma}))^\top \big) \b{\Sigma}^{\frac{1}{2}}$,
where $\nabla_E f(\b{\Sigma})$ denotes the Euclidean gradient.
\item Vector transport: the vector transport is changed to identity, i.e., $\mathcal{T}'_{\b{\Sigma}_1,\b{\Sigma}_2}(\b{\xi}') = \b{\xi}',
$ hence, optimization becomes vector transport free.
\item Exponential map: the exponential map on a SPD manifold becomes $\text{Exp}_{\b{\Sigma}}(\b{\xi}') = \b{\Sigma}^{\frac{1}{2}}\, \text{exp}(\b{\xi}')\, \b{\Sigma}^{\frac{1}{2}}$. 
\item Adjoint vector transport: the adjoint of vector transport on a SPD manifold remains the same. In other words, if before mapping we have the definition of adjoint vector transport as $g_{\b{\Sigma}_1}(\b{\xi}, \mathcal{T}^*_{\b{\Sigma}_1, \b{\Sigma}_2} \b{\eta}) = g_{\b{\Sigma}_2}(\mathcal{T}_{\b{\Sigma}_1, \b{\Sigma}_2} \b{\xi}, \b{\eta}), \forall \b{\xi} \in T_{\b{\Sigma}_1} {\mathcal{M}}, \forall \b{\eta} \in T_{\b{\Sigma}_2} {\mathcal{M}}$ \citep{ring2012optimization},
we will have $g_{\b{\Sigma}_1}(\b{\xi}', \mathcal{T}^{'^*}_{\b{\Sigma}_1, \b{\Sigma}_2} \b{\eta}') = g_{\b{\Sigma}_2}(\mathcal{T}'_{\b{\Sigma}_1, \b{\Sigma}_2} \b{\xi}', \b{\eta}'), \forall \b{\xi}' \in T_{\b{\Sigma}_1} {\mathcal{M}}, \forall \b{\eta}' \in T_{\b{\Sigma}_2} {\mathcal{M}}$.
\item
Retraction: the approximation of Euclidean retraction, using second-order Taylor expansion, on a SPD manifold becomes $\text{Ret}_{\b{\Sigma}}(\b{\xi}') = \b{\Sigma} + \b{\Sigma}^{\frac{1}{2}} \b{\xi}' \b{\Sigma}^{\frac{1}{2}} + \frac{1}{2} \b{\Sigma}^{\frac{1}{2}} \b{\xi}'^{\,2} \b{\Sigma}^{\frac{1}{2}}
$.
\end{itemize}
\end{proposition}
\begin{proof}

\noindent
\textbf{$\bullet$ Metric: }
$g_{\b{\Sigma}}(\b{\xi}', \b{\eta}') \overset{(a)}{=} \textbf{tr}(\b{\Sigma}^{-1} \b{\xi} \b{\Sigma}^{-1} \b{\eta}) = \textbf{tr}(\b{\Sigma}^{-\frac{1}{2}} \b{\Sigma}^{-\frac{1}{2}} \b{\xi} \b{\Sigma}^{-\frac{1}{2}} \b{\Sigma}^{-\frac{1}{2}} \b{\eta}) \overset{(b)}{=} \textbf{tr}(\underbrace{\b{\Sigma}^{-\frac{1}{2}} \b{\xi} \b{\Sigma}^{-\frac{1}{2}}}_{=\, \b{\xi}'} \\ \underbrace{\b{\Sigma}^{-\frac{1}{2}} \b{\eta} \b{\Sigma}^{-\frac{1}{2}}}_{=\, \b{\eta}'}) \overset{(\ref{eq_mapping})}{=} \textbf{tr}(\b{\xi}' \b{\eta}')$
where $(a)$ is because of definition of metric on SPD manifolds (see Table \ref{table_operators}) and $(b)$ is thanks to the cyclic property of trace. 

\noindent
\textbf{$\bullet$ Gradient: } 
$\nabla' f(\b{\Sigma}) \overset{(\ref{eq_mapping})}{=} \b{\Sigma}^{-\frac{1}{2}} \nabla f(\b{\Sigma}) \b{\Sigma}^{-\frac{1}{2}}  \overset{(a)}{=} \frac{1}{2} \b{\Sigma}^{-\frac{1}{2}} \b{\Sigma} \big( \nabla_E f(\b{\Sigma}) + (\nabla_E f(\b{\Sigma}))^\top \big) \b{\Sigma} \b{\Sigma}^{-\frac{1}{2}} = \frac{1}{2} \b{\Sigma}^{\frac{1}{2}} \big( \nabla_E f(\b{\Sigma}) + (\nabla_E f(\b{\Sigma}))^\top \big) \b{\Sigma}^{\frac{1}{2}}$, where $(a)$ is due to definition of gradient on a SPD manifold (see Table \ref{table_operators}). 

\noindent
\textbf{$\bullet$ Vector transport: } 
$\mathcal{T}'_{\b{\Sigma}_1,\b{\Sigma}_2}(\b{\xi}') \overset{(\ref{eq_mapping})}{=} \b{\Sigma}_2^{-\frac{1}{2}} \mathcal{T}_{\b{\Sigma}_1,\b{\Sigma}_2}(\b{\xi}) \b{\Sigma}_2^{-\frac{1}{2}} \overset{(a)}{=} \underbrace{\b{\Sigma}_2^{-\frac{1}{2}} \big( \b{\Sigma}_2^{\frac{1}{2}}}_{=\,\b{I}} \b{\Sigma}_1^{-\frac{1}{2}} \b{\xi} \b{\Sigma}_1^{-\frac{1}{2}} \underbrace{\b{\Sigma}_2^{\frac{1}{2}} \big) \b{\Sigma}_2^{-\frac{1}{2}}}_{=\,\b{I}} \\= \b{\Sigma}_1^{-\frac{1}{2}} \b{\xi} \b{\Sigma}_1^{-\frac{1}{2}} \overset{(\ref{eq_mapping})}{=} \b{\xi}'$, where $(a)$ is due to definition of vector transport on a SPD manifold (see Table \ref{table_operators}).

\noindent
\textbf{$\bullet$ Exponential map: }
$\text{Exp}_{\b{\Sigma}}(\b{\xi}') \overset{(a)}{=} \b{\Sigma}\, \text{exp}(\b{\Sigma}^{-1} \b{\xi}) \overset{(b)}{=} \b{\Sigma}^{\frac{1}{2}}\, \text{exp}(\b{\Sigma}^{-\frac{1}{2}} \b{\xi} \b{\Sigma}^{-\frac{1}{2}})\, \b{\Sigma}^{\frac{1}{2}} \overset{(\ref{eq_mapping})}{=} \b{\Sigma}^{\frac{1}{2}}\, \text{exp}(\b{\xi}') \\ \b{\Sigma}^{\frac{1}{2}}$ 
where $(a)$ is shown in {\cite[Eq. 3.3]{sra2015conic}} and $(b)$ is shown in {\cite[Eq. 3.2]{sra2015conic}}. Also see Table \ref{table_operators}.

\noindent
\textbf{$\bullet$ Retraction: }
$\text{Ret}_{\b{\Sigma}}(\b{\xi}') \overset{(a)}{=} \b{\Sigma} + \b{\xi} + \frac{1}{2} \b{\xi} \b{\Sigma}^{-1} \b{\xi} \overset{(\ref{eq_mapping})}{=} \b{\Sigma} + \b{\Sigma}^{\frac{1}{2}}\, \b{\xi}'\, \b{\Sigma}^{\frac{1}{2}} + \frac{1}{2} \b{\Sigma}^{\frac{1}{2}} \b{\xi}' \underbrace{\b{\Sigma}^{\frac{1}{2}} \b{\Sigma}^{-1} \b{\Sigma}^{\frac{1}{2}}}_{=\, \b{I}} \b{\xi}' \b{\Sigma}^{\frac{1}{2}} = \b{\Sigma} + \b{\Sigma}^{\frac{1}{2}} \b{\xi}' \b{\Sigma}^{\frac{1}{2}} + \frac{1}{2} \b{\Sigma}^{\frac{1}{2}} \b{\xi}'^{\,2} \b{\Sigma}^{\frac{1}{2}}$, 
where $(a)$ is because of approximation of Euclidean retraction, using the second-order Taylor expansion, on SPD manifolds (see Table \ref{table_operators}).
\end{proof}

\section{Vector Transport Free Riemannian LBFGS Using Mapping by Cholesky Decomposition}\label{section_VTF_Cholesky}

Our second proposed mapping is by Cholesky decomposition which is very efficient computationally. Consider the Cholesky decomposition of point $\b{\Sigma} \in \mathcal{M}$ \citep{golub2013matrix}:
\begin{align}
& \b{0} \prec \b{\Sigma} = \b{L} \b{L}^\top \implies \b{\Sigma}^{-1} = \b{L}^{-\top} \b{L}^{-1}, \label{eq_Cholesky_decomposition} 
\end{align}
where $\b{L} \in \mathbb{R}^{n \times n}$ is the lower-triangular matrix in Cholesky decomposition. 
It is noteworthy that many of the MATLAB matrix multiplication operators, which the Manopt toolbox \citep{boumal2014manopt} also uses, apply Cholesky decomposition internally due to its efficiency. 

In the tangent space of every point $\b{\Sigma} \in \mathcal{M}$, the proposed mapping is:
\begin{align}
&\b{\xi}' := \b{L}^{-1}\, \b{\xi}\, \b{L}^{-\top} \implies \b{\xi} = \b{L}\, \b{\xi}'\, \b{L}^{\top}, \label{eq_mapping_Cholesky}
\end{align}
where $\b{\xi}, \b{\xi}' \in T_{\b{\Sigma}} {\mathcal{M}} \subset \mathbb{S}_{++}^n$.
Note that, similar to the previous mapping, the tangent matrix is still symmetric under this mapping; hence, it remains in the tangent space of the SPD manifold \citep{bhatia2009positive}.
Similar to the previous mapping, under the second proposed mapping (\ref{eq_mapping}), the Riemannian operators on SPD manifold are simplified. These operators can be found in Table \ref{table_operators}. In the following, we provide proofs for these operators. 

\begin{proposition}\label{equation_adjoint_vector_transport_mapping_Cholesky}
After mapping (\ref{eq_mapping_Cholesky}), we have:
\begin{itemize}
\item Metric: the metric on a SPD manifold is reduced to the Euclidean inner product, i.e., $g_{\b{\Sigma}}(\b{\xi}', \b{\eta}') = \textbf{tr}(\b{\xi}' \b{\eta}')
$.
\item Gradient: the gradient on SPD manifold is changed to $\nabla' f(\b{\Sigma}) = \frac{1}{2} \b{L}^{\top} \big( \nabla_E f(\b{\Sigma}) + (\nabla_E f(\b{\Sigma}))^\top \big) \b{L}
$, where $\nabla_E f(\b{\Sigma})$ denotes the Euclidean gradient.
\item Vector transport: the vector transport is changed to identity, i.e., $\mathcal{T}'_{\b{\Sigma}_1,\b{\Sigma}_2}(\b{\xi}') = \b{\xi}'$,
hence, optimization becomes vector transport free. 
\item Exponential map: the exponential map becomes $\text{Exp}_{\b{\Sigma}}(\b{\xi}') = \b{\Sigma}\, \text{exp}(\b{L}^{-\top} \b{\xi}' \b{L}^\top)$. 
\item Adjoint vector transport: the adjoint vector transport becomes $g_{\b{\Sigma}_1}(\b{\xi}', \mathcal{T}^{'^*}_{\b{\Sigma}_1, \b{\Sigma}_2} \b{\eta}') = g_{\b{\Sigma}_2}(\mathcal{T}'_{\b{\Sigma}_1, \b{\Sigma}_2} \b{\xi}', \b{\eta}'), \forall \b{\xi}' \in T_{\b{\Sigma}_1} {\mathcal{M}}, \forall \b{\eta}' \in T_{\b{\Sigma}_2} {\mathcal{M}}$ as we had in Proposition \ref{proposition_mapping_inverseSecondRoot}. 
\item Retraction: the approximation of Euclidean retraction by second-order Taylor expansion on a SPD manifold becomes $\text{Ret}_{\b{\Sigma}}(\b{\xi}') = \b{\Sigma} + \b{L}\, \b{\xi}' \b{L}^{\top} + \frac{1}{2} \b{L}\, \b{\xi}'^{\,2} \b{L}^{\top}$.
\end{itemize}
\end{proposition}
\begin{proof}

\noindent
\textbf{$\bullet$ Metric: }
$g_{\b{\Sigma}}(\b{\xi}', \b{\eta}') \overset{(a)}{=} \textbf{tr}(\b{\Sigma}^{-1} \b{\xi} \b{\Sigma}^{-1} \b{\eta}) = \textbf{tr}(\b{L}^{-\top} \b{L}^{-1} \b{\xi} \b{L}^{-\top} \b{L}^{-1} \b{\eta}) \overset{(b)}{=} \textbf{tr}(\underbrace{\b{L}^{-1} \b{\xi} \b{L}^{-\top}}_{=\, \b{\xi}'} \\
\underbrace{\b{L}^{-1} \b{\eta} \b{L}^{-\top}}_{=\, \b{\eta}'}) \overset{(\ref{eq_mapping_Cholesky})}{=} \textbf{tr}(\b{\xi}' \b{\eta}')$,
where $(a)$ is because of definition of metric on SPD manifolds (see Table \ref{table_operators}) and $(b)$ is thanks to the cyclic property of trace.

\noindent
\textbf{$\bullet$ Gradient: }
$\nabla' f(\b{\Sigma}) \overset{(\ref{eq_mapping_Cholesky})}{=} \b{L}^{-1} \nabla f(\b{\Sigma}) \b{L}^{-\top} \overset{(a)}{=} \frac{1}{2} \b{L}^{-1} \b{\Sigma} \big( \nabla_E f(\b{\Sigma}) + (\nabla_E f(\b{\Sigma}))^\top \big) \b{\Sigma} \b{L}^{-\top} \\
\overset{(\ref{eq_Cholesky_decomposition})}{=} \frac{1}{2} \b{L}^{-1} \b{L} \b{L}^\top \big( \nabla_E f(\b{\Sigma}) + (\nabla_E f(\b{\Sigma}))^\top \big) \b{L} \b{L}^\top \b{L}^{-\top} = \frac{1}{2} \b{L}^\top \big( \nabla_E f(\b{\Sigma}) + (\nabla_E f(\b{\Sigma}))^\top \big) \b{L}$, 
where $(a)$ is due to definition of gradient on a SPD manifold (see Table \ref{table_operators}). 

\noindent
\textbf{$\bullet$ Vector transport: }
$\mathcal{T}'_{\b{\Sigma}_1,\b{\Sigma}_2}(\b{\xi}') \overset{(\ref{eq_mapping_Cholesky})}{=} \b{L}_2^{-1} \mathcal{T}_{\b{\Sigma}_1,\b{\Sigma}_2}(\b{\xi}) \b{L}_2^{-\top} \overset{(a)}{=} {\b{L}_2^{-1} \big( \b{L}_2}\b{L}_1^{-1} \b{\xi} \b{L}_1^{-\top} {\b{L}_2^{\top} \big) \b{L}_2^{-\top}} \\
= \b{L}_1^{-1} \b{\xi} \b{L}_1^{-\top} \overset{(\ref{eq_mapping_Cholesky})}{=} \b{\xi}'$,
where $(a)$ is due to definition of vector transport on SPD manifold (see Table \ref{table_operators}).


\noindent
\textbf{$\bullet$ Exponential map: }
The exponential map is changed to  $\text{Exp}_{\b{\Sigma}}(\b{\xi}') \overset{(a)}{=} \b{\Sigma}\, \text{exp}(\b{\Sigma}^{-1} \b{\xi}) \overset{(b)}{=} \b{\Sigma}\, \text{exp}(\b{L}^{-\top} \b{L}^{-1} \b{L}\, \b{\xi}' \b{L}^\top)
= \b{\Sigma}\, \text{exp}(\b{L}^{-\top} \b{\xi}' \b{L}^\top)$,
where $(a)$ is shown in {\cite[Eq. 3.3]{sra2015conic}} and $(b)$ is because of Eqs. (\ref{eq_Cholesky_decomposition}) and (\ref{eq_mapping_Cholesky}).

\noindent
\textbf{$\bullet$ Retraction: }
$\text{Ret}_{\b{\Sigma}}(\b{\xi}') \overset{(a)}{=} \b{\Sigma} + \b{\xi} + \frac{1}{2} \b{\xi} \b{\Sigma}^{-1} \b{\xi} \overset{(b)}{=} \b{\Sigma} + \b{L}\, \b{\xi}' \b{L}^{\top} + \frac{1}{2} \b{L}\, \b{\xi}' \b{L}^{\top} \b{L}^{-\top} \b{L}^{-1} \b{L} \b{\xi}' \b{L}^{\top} = \b{\Sigma} + \b{L}\, \b{\xi}' \b{L}^{\top} + \frac{1}{2} \b{L}\, \b{\xi}'^{\,2} \b{L}^{\top}$, where $(a)$ is because of approximation of Euclidean retraction, using second-order Taylor expansion, on SPD manifolds (see Table \ref{table_operators}), and $(b)$ is because of Eqs. (\ref{eq_Cholesky_decomposition}) and (\ref{eq_mapping_Cholesky}). 
\end{proof}

Noticing Eq. (\ref{eq_mapping_Cholesky}), the approximation of retraction under mapping by Cholesky decomposition (see Proposition \ref{equation_adjoint_vector_transport_mapping_Cholesky}), can be restated as $\text{Ret}_{\b{\Sigma}}(\b{\xi}') = 0.5\, \b{\Sigma} + 0.5\, \b{L} (\b{I} + \b{\xi}')^2 \b{L}^\top$. Defining $\b{\Psi} := \b{L} (\b{I} + \b{\xi}')$ restates this retraction as $\text{Ret}_{\b{\Sigma}}(\b{\xi}') = 0.5\, \b{\Sigma} + 0.5\, \b{\Psi} \b{\Psi}^\top$ because $\b{\xi}'$ is symmetric. The term $\b{\Psi} \b{\Psi}^\top$ is very efficient and fast to compute because it is symmetric.

\section{Analytical Discussion and Complexity Analysis}

\begin{corollary}\label{corollary_transforms_isometric}
Propositions \ref{proposition_mapping_inverseSecondRoot} and \ref{equation_adjoint_vector_transport_mapping_Cholesky} show that under mapping (\ref{eq_mapping}) or (\ref{eq_mapping_Cholesky}), both vector transport and adjoint vector transport are identity. As these transforms become identity, they also become isometric because inner products of vectors do not change under these transforms. As they are identity, these transforms also preserve the length of vectors under the proposed mappings. 
\end{corollary}

Propositions \ref{proposition_mapping_inverseSecondRoot} and \ref{equation_adjoint_vector_transport_mapping_Cholesky} and Corollary \ref{corollary_transforms_isometric} show the two proposed mappings simplify vector transport and adjoint vector transport to isometric identity and reduce the Riemannian metric to Euclidean inner product. These reductions and simplifications reduce computations significantly during optimization on the manifold. 
The VTF-RLBFGS algorithm using either of the proposed mappings is shown in Algorithm \ref{algorithm_VTF_RLBFGS}.
As the algorithm shows, at every new point $\b{\Sigma}_k$, the entire parameters are computed in the paradigm of mapping because the manifold operators, calculated as in Table \ref{table_operators}, are in that paradigm. This simplifies operators such as metric and removes vector transports from RLBFGS. 
In case the Riemannian gradient is given directly by the user to RLBFGS, the Riemannian gradient, which is in the tangent space, should be mapped explicitly by Eqs. (\ref{eq_mapping}) and (\ref{eq_mapping_Cholesky}) at every iteration. However, if the Riemannian gradient is calculated from the Euclidean gradient, it should not be mapped explicitly, since it is already in the paradigm of mapping implicitly because of the used operators of Table \ref{table_operators} in that paradigm.

\SetAlCapSkip{0.5em}
\IncMargin{0.8em}
\begin{algorithm2e}[!t]
\DontPrintSemicolon
    \textbf{Input}: Initial point $\b{\Sigma}_0$\;
    $\b{H}_0 := \frac{1}{\sqrt{{g}'_{\b{\Sigma}_0}(\nabla' f(\b{\Sigma}_0), \nabla' f(\b{\Sigma}_0))}} \b{I}$\;
    \For{$k = 0,1,\dots$}{
        Compute $\nabla' f(\b{\Sigma}_k)$ from Euclidean gradient by one of the mappings in Table \ref{table_operators}\;
        $\b{\xi}'_k := \text{GetDirection}(-\nabla' f(\b{\Sigma}_k), k)$\;
        $\alpha_k := $ Line search with Wolfe conditions\;
        $\b{\Sigma}_{k+1} := \text{Exp}_{\b{\Sigma}_k}(\alpha_k \b{\xi}'_k) \,\, \text{or} \,\, \text{Ret}_{\b{\Sigma}_k}(\alpha_k \b{\xi}'_k)$\;
        $\b{s}'_{k+1} := \alpha_k \b{\xi}'_k$\;
        $\b{y}'_{k+1} := \nabla' f(\b{\Sigma}_{k+1}) - \nabla' f(\b{\Sigma}_{k})$\;
        $\b{H}_{k+1} := \frac{g'_{\b{\Sigma}_{k+1}}(\b{s}'_{k+1}, \b{y}'_{k+1})}{g'_{\b{\Sigma}_{k+1}}(\b{y}'_{k+1}, \b{y}'_{k+1})}$\;
        Store $\b{y}'_{k+1}$, $\b{s}'_{k+1}$, $g'_{\b{\Sigma}_{k+1}}(\b{s}'_{k+1}, \b{y}'_{k+1})$, $g'_{\b{\Sigma}_{k+1}}(\b{s}'_{k+1}, \b{s}'_{k+1})$, and $\b{H}_{k+1}$\;
    }
    \textbf{return} $\b{\Sigma}_{k+1}$\;
    \;
    \textbf{Function} $\text{GetDirection}(\b{p}', k)$\;
    \uIf{$k > 0$}{
        $\rho_k := \frac{1}{\b{g}'_{\b{\Sigma}_k}(\b{y}'_k, \b{s}'_k)}$\;
        $\b{\tilde{p}}' := \b{p}' - \rho_k\, \b{g}'_{\b{\Sigma}_k}(\b{s}'_k, \b{p}')  \b{y}'_k$\;
        $\b{\widehat{p}}' := \text{GetDirection}(\b{\tilde{p}}', k-1)$\;
        \textbf{return} $\b{\widehat{p}}' - \rho_k\, \b{g}'_{\b{\Sigma}_k}(\b{y}'_k, \b{\widehat{p}}'_k) \b{s}'_k + \rho_k\, \b{g}'_{\b{\Sigma}_k}(\b{s}'_k, \b{s}'_k) \b{p}'$\;
    }
    \Else{
        \textbf{return} $\b{H}_0\, \b{p}'$\;
    }
\caption{The VTF-RLBFGS algorithm}\label{algorithm_VTF_RLBFGS}
\end{algorithm2e}
\DecMargin{0.8em}

\begin{lemma}
Vector transport is valid under both proposed mappings (\ref{eq_mapping}) and (\ref{eq_mapping_Cholesky}) because they preserve the properties of vector transport. 
\end{lemma}
\begin{proof}
A valid vector transport should satisfy three properties \citep{hosseini2020alternative} (also see {\cite[Definition 8.1.1]{absil2009optimization}} and {\cite[Definition 10.62]{boumal2020introduction}}): 

(1) The following property $\exists \b{v} \in T_{\b{\Sigma}_1}\mathcal{M}: \mathcal{T}_{\b{\Sigma}_1, \b{\Sigma}_2}(\b{\xi}) \in  T_{\text{Ret}_{\b{\Sigma}}(\b{v})}\mathcal{M}, \forall \b{\xi} \in T_{\b{\Sigma}_1}\mathcal{M}$
holds because the tangent space $T_{\text{Ret}_{\b{\Sigma}}(\b{v})}\mathcal{M}$ is isomorphic to the space of symmetric matrices and the identity vector transports return the tangent vector itself which is in the space of symmetric matrices. 

(2) The vector transport of a tangent vector from one point to itself should be the same tangent vector. This holds because vector transport is identity under the proposed mappings: $\mathcal{T}_{\b{\Sigma}, \b{\Sigma}} (\b{\xi}) = \b{\xi}, \forall \b{\xi} \in T_{\b{\Sigma}_1}\mathcal{M}$.

(3) The vector transport $\mathcal{T}_{\b{\Sigma}_1, \b{\Sigma}_2} (\b{\xi})$ should be linear which is because it is equal to $\b{\xi}$ under the proposed mappings.

As after applying the mappings, the vector transport holds the three above properties, it is a valid transport. Q.E.D.
\end{proof}

\begin{proposition}
The time complexity of the recursion part in RLBFGS is improved from $\Theta(m n^3)$ to $\Theta(m n^2)$ after mapping (\ref{eq_mapping}) or (\ref{eq_mapping_Cholesky}), where $m$ is the memory limit, i.e., the maximum number of recursions in RLBFGS (proof is available in Supplementary Material).
This time improvement shows off better in problems whose computation of gradients in Eq. (\ref{eq_RLBFGS_y}), or line $\b{y}'_{k+1}$ in Algorithm \ref{algorithm_VTF_RLBFGS}, is not dominant in complexity. 
\end{proposition}

\section{Simulations}\label{section_simulations}

In this section, we evaluate the effectiveness of the proposed mappings, i.e. (\ref{eq_mapping}) and (\ref{eq_mapping_Cholesky}), in the tangent space. Here, we show that these mappings often improve the performance and speed of RLBFGS. 
The code of this article is available in \url{https://github.com/bghojogh/LBFGS-Vector-Transport-Free}.
In our reports, we denote the proposed Vector Transport Free (VTF) RLBFGS with VTF-RLBFGS where ISR and Cholesky (or Chol.) stand for VTF-RLBFGS using mapping by inverse second root and Cholesky decomposition, respectively. 
For RLBFGS, with and without the proposed mappings, we use both Wolfe linesearch conditions. 
The programming language, used for experiments, was MATLAB and the hardware was Intel Core-i7 CPU with the base frequency 2.20 GHz and 12 GB RAM.
For every experiment, we performed optimization for ten runs and the reported results are the average of performances over the runs. We evaluated our mappings with various application problems, explained below.

\subsection{Gaussian Mixture Model}


\noindent
\textbf{$\bullet$ Formulation:}
An optimization problem, which we selected for evaluation, is the Riemannian optimization for Gaussian Mixture Model (GMM) without the use of expectation maximization. We employ RLBFGS with and without the proposed mappings for fitting the GMM problem whose algorithm can be found in \citep{hosseini2020alternative,hosseini2015mixest}. 
This is a suitable problem for evaluation of the proposed mappings because the covariance matrices are SPD \citep{bhatia2009positive}. 
For this, we minimize the negative log-likelihood of GMM where the covariance matrix is constrained to belong to the SPD matrix manifold \citep{hosseini2020alternative}. For $n$-dimensional GMM, the optimization problem is:
\begin{equation}\label{eq_GMM_problem}
\begin{aligned}
& \underset{\{\alpha_j,\, \b{\mu}_j,\, \b{\Sigma}_j\}_{j=1}^K}{\text{minimize}}
& & -\sum_{i=1}^N \log\Big(\sum_{j=1}^K \alpha_j\, \mathcal{N}(\b{x}_i; \b{\mu}_j, \b{\Sigma}_j)\Big), \\
& \text{subject to}
& & \b{\Sigma}_j \in \mathcal{M} = \mathbb{S}_{++}^n, \quad \forall j \in \{1, \dots, K\},
\end{aligned}
\end{equation}
where $N$ denotes the sample size, $K$ denotes the number of components in mixture model, and $\alpha_j$, $\b{\mu}_j$, and $\b{\Sigma}_j$ are the mixing probability, mean, and covariance of the $j$-th component, respectively. We use the same reformulation trick of \citep{hosseini2020alternative} to reformulate the cost function  of (\ref{eq_GMM_problem}).
The mixture parameters were initialized using K-means++ \citep{arthur2007k} following \citep{hosseini2020alternative}.
Three different levels of separation of Gaussian models, namely low, mid, and high, were used. The reader can refer to \citep{hosseini2020alternative} for mathematical details of these separation levels.

\begin{table*}[!t]
\caption{Comparison of average results over ten runs where exponential map is used in algorithms and $K\in \{2,5\}$, $n\in \{2,10\}$, $N=10n^2 \in \{40, 1000\}$. The $\#$iters, conv, iter, diff, and std are short for number of iterations, convergence, iteration, difference, and standard deviation, respectively.}
\label{table_RiemMax_10n2_K2_expm}
\renewcommand{\arraystretch}{1.3}  
\centering
\scalebox{0.7}{    
\renewcommand{\arraystretch}{1.1}
\begin{tabular}{l | l | l | l | l | l | l | l | l}
\hline
\hline
$K$ & $n$ & Separation & Algorithm & \#iters & conv. time & time diff. std & iter. time & last cost \\
\hline
\multirow{18}{*}{2} & \multirow{9}{*}{2} & \multirow{ 3}{*}{Low} 	
& 	  VTF (ISR) &	 53.100$\pm$18.248 &	 68.380$\pm$52.834 & 13.405 &	 1.140$\pm$0.504 & 0.364$\pm$0.444 \\
& & & VTF (Chol.) &	 52.100$\pm$17.866 &	 61.096$\pm$48.433 & 17.443	& 1.030$\pm$0.463 & 0.364$\pm$0.444 \\
& & & RLBFGS 	&	 52.500$\pm$16.595 &	 65.890$\pm$49.208 & -- &	 1.125$\pm$0.458 & 0.364$\pm$0.444 \\
\cline{3-9}    
& & \multirow{ 3}{*}{Mid} 
&     VTF (ISR) &	 56.400$\pm$21.813 &	 76.124$\pm$69.189 &	10.712 & 1.150$\pm$0.574 & 0.657$\pm$0.344 \\
& & & VTF (Chol.) &	 54.400$\pm$19.156 &	 68.456$\pm$64.290 & 48.016 &	 1.099$\pm$0.504 & 0.638$\pm$0.333 \\
& & & RLBFGS 	&	 57.700$\pm$19.844 &	 81.957$\pm$64.463 & -- &	 1.250$\pm$0.550 & 0.657$\pm$0.344 \\
\cline{3-9}   
& & \multirow{ 3}{*}{High} 
& 	  VTF (ISR) &	 25.500$\pm$3.064 &	 9.518$\pm$2.922 & 0.333	 &	 0.366$\pm$0.067  & 0.341$\pm$0.371 \\
& & & VTF (Chol.) &	 26.000$\pm$4.738 &	 10.254$\pm$5.468  & 3.132 &	 0.377$\pm$0.108 & 0.341$\pm$0.371 \\
& & & RLBFGS &	     25.500$\pm$3.064 &	 10.054$\pm$3.141  & -- &	 0.386$\pm$0.073 & 0.341$\pm$0.371 \\
\cline{2-9}
& \multirow{9}{*}{10} & \multirow{ 3}{*}{Low} 	
& 	  	  VTF (ISR) &     77.600$\pm$54.175 &	 235.615$\pm$466.119 & 12.040 &	 1.936$\pm$1.751 & 4.210$\pm$0.889		\\
& & 	& VTF (Chol.) &   84.100$\pm$73.260 &	 313.398$\pm$714.908 & 257.528 &	 2.041$\pm$2.150 & 4.209$\pm$0.889   	\\
& & 	& RLBFGS &   	  77.600$\pm$52.754 &	 242.743$\pm$458.641 & -- &	 2.069$\pm$1.733 & 4.209$\pm$0.889       \\
\cline{3-9}      
& &\multirow{ 3}{*}{Mid}  	
& 		  VTF (ISR) &     44.600$\pm$7.260 &	 43.654$\pm$16.872 & 4.139 &	 0.948$\pm$0.208  & 4.262$\pm$1.098         \\
& & 	& VTF (Chol.) &   45.900$\pm$8.647 &	 45.661$\pm$20.088 & 5.615 &	 0.955$\pm$0.236  & 4.262$\pm$1.098     \\
& & 	& RLBFGS &   	  45.200$\pm$8.080 &	 48.104$\pm$19.981 & -- &	 1.025$\pm$0.243  & 4.262$\pm$1.098         \\
\cline{3-9}                                
& &\multirow{ 3}{*}{High}  	
& 		  VTF (ISR) &     44.400$\pm$9.252 &	 43.684$\pm$22.460 & 11.746 &	 0.936$\pm$0.254 & 3.874$\pm$1.395          \\
& & 	& VTF (Chol.) &   47.100$\pm$8.333 &	 48.278$\pm$21.112 & 8.992 &	 0.987$\pm$0.240 & 3.874$\pm$1.395      \\
& & 	& RLBFGS &   	  43.300$\pm$7.150 &	 43.904$\pm$17.626 & -- &	 0.981$\pm$0.225 & 3.874$\pm$1.395          \\
\hline
\hline
\multirow{18}{*}{5} & \multirow{9}{*}{2} & \multirow{ 3}{*}{Low} 	
& 		  VTF (ISR) 	& 152.400$\pm$62.819 	& 1344.573$\pm$999.949 & 649.595	 &7.560$\pm$3.406 & 0.260$\pm$0.458 \\
& & 	& VTF (Chol.) 	& 174.800$\pm$114.139 &	 2168.886$\pm$3090.820 & 2347.949 &	 8.680$\pm$6.348 & 0.265$\pm$0.466 \\
& & 	& RLBFGS 		& 166.300$\pm$76.884 	& 1767.676$\pm$1406.454 & -- 	 &8.788$\pm$4.427 & 0.267$\pm$0.454 \\
\cline{3-9} 
& & \multirow{ 3}{*}{Mid} 	
& 		  VTF (ISR) 	& 120.700$\pm$69.620 	& 942.713$\pm$1063.075 & 1120.404 &	 5.807$\pm$3.877 & 0.781$\pm$0.207 \\
& & 	& VTF (Chol.) 	& 121.600$\pm$65.030 	& 897.110$\pm$988.826  &	1315.410 &     5.720$\pm$3.445 & 0.781$\pm$0.207 \\
& & 	& RLBFGS 		& 136.300$\pm$91.493 	& 1404.654$\pm$1986.798& -- & 	 7.057$\pm$5.378 & 0.764$\pm$0.225 \\
\cline{3-9} 
& & \multirow{ 3}{*}{High}  	
& 		  VTF (ISR) 	& 46.400$\pm$17.322 	& 94.741$\pm$105.045 	& 22.071 & 1.739$\pm$0.902 & 1.805$\pm$0.385 \\
& & 	& VTF (Chol.) 	& 46.200$\pm$19.803 	& 98.065$\pm$122.950 & 6.436	& 1.729$\pm$1.020 & 1.805$\pm$0.385 \\
& & 	& RLBFGS 		& 46.200$\pm$18.937 	& 104.934$\pm$126.734 & -- &	 1.879$\pm$1.064 & 1.805$\pm$0.385 \\
\cline{2-9}  
& \multirow{9}{*}{10} & \multirow{ 3}{*}{Low} 	
& 		  VTF (ISR) 	& 295.900$\pm$86.577 	& 5524.495$\pm$3173.855 & 2664.627 	 &17.299$\pm$5.221 & 6.175$\pm$0.745 \\
& & 	& VTF (Chol.) 	& 292.300$\pm$83.828 	& 5294.565$\pm$3165.735 & 5106.693 	 &16.737$\pm$5.350 & 6.197$\pm$0.718 \\
& & 	& RLBFGS 		& 318.800$\pm$100.216 &	 7083.082$\pm$4801.718 & -- &	 20.279$\pm$6.859 & 6.173$\pm$0.744 \\
\cline{3-9}                          
& & \multirow{ 3}{*}{Mid} 	
& 		  VTF (ISR) 	& 134.200$\pm$54.956 	& 1139.332$\pm$919.917 & 306.469 &	 7.302$\pm$3.227 & 6.753$\pm$0.708  \\
& & 	& VTF (Chol.) 	& 133.300$\pm$52.415 	& 1100.519$\pm$842.840 & 296.996 &	 7.183$\pm$3.045 & 6.753$\pm$0.708  \\
& & 	& RLBFGS 		& 135.300$\pm$54.965 	& 1268.707$\pm$967.678 & -- &	 8.070$\pm$3.580 & 6.753$\pm$0.708  \\
\cline{3-9}                           
& & \multirow{ 3}{*}{High}  	
& 		  VTF (ISR) 	& 68.600$\pm$12.367 	& 241.487$\pm$87.593 	& 18.122 & 3.398$\pm$0.764 & 6.599$\pm$0.836  \\
& & 	& VTF (Chol.) 	& 74.000$\pm$14.071 	& 279.271$\pm$105.828 & 40.158 &	 3.632$\pm$0.838 & 6.599$\pm$0.836  \\
& & 	& RLBFGS 		& 68.300$\pm$11.982 	& 258.475$\pm$93.959 & -- 	& 3.661$\pm$0.790 & 6.599$\pm$0.836  \\
\hline
\hline
\end{tabular}%

}
\end{table*}


\noindent
\textbf{$\bullet$ Results:}
We compared RLBFGS performances with and without our proposed mappings. 
The average number of iterations, convergence time, time per iteration, and the cost value of last iteration are reported in Table \ref{table_RiemMax_10n2_K2_expm} where exponential map is used for retraction. 
Results for Taylor approximation of exponential map used as retraction is available in Table 1 of supplementary material.
In these experiments, we report performances for $K\in \{2,5\}$,  $n\in \{2,10\}$, $N=10n^2 \in \{40, 1000\}$. 
For the sake of fairness, all the three compared algorithms start with the same initial points in every run.
The reader can see more extensive experiments for more sample size and dimensionality, i.e. $K\in \{2,5\}$,  $n\in \{2,10, 100\}$, $N=100n^2 \in \{400, 10000, 1000000\}$, in the Supplementary Material.
In these tables, we have also provided the standard deviation (std) of time differences between VTF-RLBFGS and RLBFGS, over the ten runs. This value shows how much the average convergence time improvements over RLBFGS are reliable.

\noindent
\textbf{$\bullet$ Discussion by Varying Separability:}
As Table \ref{table_RiemMax_10n2_K2_expm} and the Table 1 of Supplementary Material show, the proposed ISR mapping converges faster than no mapping most often in all separability levels. Both its time per iteration and number of iterations are often less than no mapping.
These tables also show that the proposed Cholesky mapping converges faster than no mapping in most of the cases, although not all cases. 
Overall, we see that the two proposed mappings make RLBFGS faster and more efficient most often. This pacing improvement can be noticed more for low and mid separability levels because they are harder cases to converge. 
In terms of quality of local minimum, the proposed mappings often find the same local minimum as no mapping, but in a faster way. The equality of the found local optima in the three methods is because of using the same RLBFGS algorithm as their base in addition to having the same initial points. 
In some cases, the proposed mappings have even found better local minima. 
Moreover, as expected, the time difference std reduces in higher separability which is a simpler task. 


\noindent
\textbf{$\bullet$ Discussion by Varying Dimensionality:}
We can discuss the results of Table \ref{table_RiemMax_10n2_K2_expm} and the Table 1 of Supplementary Material by varying dimensionality, $n \in \{2, 10\}$. 
Tables 2 and 3 of Supplementary Material also report performance for dimensions $n \in \{2, 10, 100\}$ and larger sample size.
Obviously, by increasing dimensionality, the time of convergence goes up and the faster pacing of the proposed mappings is noticed more, compared to no mapping.

\noindent
\textbf{$\bullet$ Discussion by Varying the Number of Components:}
The results of Table \ref{table_RiemMax_10n2_K2_expm} and the Table 1 of Supplementary Material can also be interpreted based on varying the number of mixture components, $K \in \{2, 5\}$. The more number of components makes the optimization problem harder. Hence, the difference of speeds of the proposed mappings and no mapping can be noticed more for $K=5$; although, for both $K$ values, the proposed mappings are often faster than no mapping. 

\noindent
\textbf{$\bullet$ Discussion by Varying Retraction Type:}
We can also compare the performance of algorithms in terms of type of retraction. Table \ref{table_RiemMax_10n2_K2_expm} and the Table 1 of Supplementary Material report performances where exponential map and Taylor approximation of exponential map are used for retraction, respectively. Comparing these tables shows that using Taylor approximation usually converges with less number of iterations compared to using exponential map. It makes sense because exponential map requires passing on geodesics. This difference of pacing is more obvious for larger number of components, i.e., $K=5$. 
Our two proposed mappings outperform no mapping for both types retraction. This shows that our mappings are effective regardless of the details of operators.

\noindent
\textbf{$\bullet$ Discussion by Varying the Sample Size:}
Table \ref{table_RiemMax_10n2_K2_expm} and the Table 1 of Supplementary Material report for $n\in \{2,10\}$, $N=10n^2 \in \{40, 1000\}$.
More experiments for larger sample size and dimensionality, i.e. $n\in \{2,10, 100\}$, $N=100n^2 \in \{400, 10000, 1000000\}$, can be found in Tables 2 and 3 in the Supplementary Material. 
Comparing those tables with Table \ref{table_RiemMax_10n2_K2_expm} and the Table 1 of Supplementary Material shows that larger sample size and/or dimensionality takes more time to converge as expected. Still, our proposed mappings often converge faster than no mapping. The difference of pacing is mostly less in larger sample size compared to smaller sample size. This is because very large sample size consumes time on computation of cost function and the difference is not given much chance to show off in that case. 


\begin{table*}[!t]
\caption{Comparison of average results over ten runs  with various algorithms where exponential map is used in algorithms and $K\in \{2,5\}$, $n = 2$, $N=1000n^2 = 4000$.}
\label{table_RiemMax_comparison_with_other_algorithms}
\renewcommand{\arraystretch}{1.3}  
\centering
\scalebox{0.7}{    
\renewcommand{\arraystretch}{1.1}
\begin{tabular}{l | l | l | l | l | l | l}
\hline
\hline
& \multicolumn{2}{c}{Low separation} & \multicolumn{2}{|c}{Mid separation} & \multicolumn{2}{|c}{High separation} \\
\hline
Algorithm & \#iters ($K=2$) & \#iters ($K=5$) & \#iters ($K=2$) & \#iters ($K=5$) & \#iters ($K=2$) & \#iters ($K=5$) \\
\hline
VTF (ISR) & 51.500$\pm$8.885 &  123.400$\pm$26.069	& 54.100$\pm$20.464 & 106.500$\pm$39.328 & 24.500$\pm$4.353 &  35.500$\pm$6.996  \\
VTF (Chol.) &	54.200$\pm$11.263 & 123.000$\pm$35.065 &	58.800$\pm$24.943 & 107.500$\pm$44.490 &	25.100$\pm$4.149 & 35.900$\pm$7.622 \\
RLBFGS 	& 52.900$\pm$8.825 & 147.400$\pm$35.926	& 57.900$\pm$29.622	& 110.200$\pm$38.064 &	24.400$\pm$4.006 &  35.600$\pm$7.516 \\
CG 	& 83.100$\pm$22.684 & 142.200$\pm$44.236	& 69.500$\pm$40.175	&  155.100$\pm$48.732 & 28.500$\pm$9.192	&  51.600$\pm$16.392 \\
EM 	& 94.400$\pm$40.114	& 277.900$\pm$142.549 & 118.100$\pm$89.794 &  224.000$\pm$101.576 & 3.200$\pm$0.422 &  3.600$\pm$0.966 \\
\hline    
\hline
\end{tabular}%

}
\end{table*}

\begin{figure}[!t]
\centering
\includegraphics[width=6in]{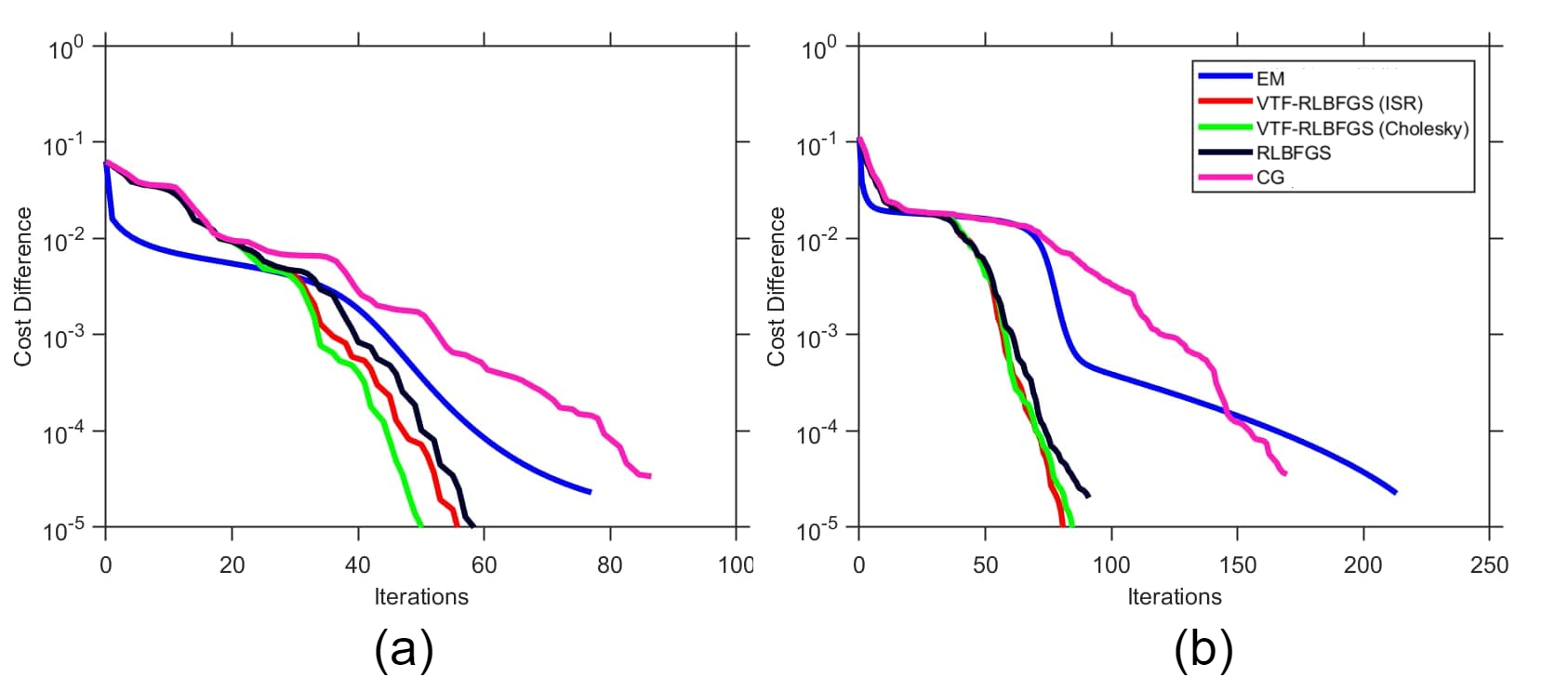}
\caption{The cost difference progress for several runs: (a) $K=2$, $n=2$, $N=4000$, low separation, and (b) $K=5$, $n=2$, $N=4000$, mid separation.}
\label{figure_cost_some_plots}
\end{figure}

\noindent
\textbf{$\bullet$ Comparison with Other Algorithms:}
We compared RLBFGS, with and without the proposed mappings, with some other algorithms, i.e., nonlinear Conjugate Gradient (CG) and Expectation Maximization (EM) for fitting GMM. 
The log-scale cost difference progress of several insightful runs are illustrated in Fig. \ref{figure_cost_some_plots}.
The average number of iterations in the algorithms are compared in Table \ref{table_RiemMax_comparison_with_other_algorithms}.
A complete set of plots for cost difference progress can be seen in Figs. 1, 2, 3, and 4 in the Supplementary Material. Figs. \ref{figure_cost_some_plots}-a and \ref{figure_cost_some_plots}-b show that in some cases, ISR mapping is faster than the Cholesky mapping and in some other cases, we have the other way around. 
As Fig. \ref{figure_cost_some_plots} and Table \ref{table_RiemMax_comparison_with_other_algorithms} show, our proposed mappings are outperforming CG and EM in the number of iterations.

\subsection{Geometric Metric Learning}

\textbf{$\bullet$ Formulation:}
As another application, we experiment with geometric metric learning. We implemented an iterative version of \citep{zadeh2016geometric} whose regularized problem is:
\begin{equation*}\label{eq_metric_learning_problem}
\begin{aligned}
& \underset{\b{W}}{\text{min.}}
& & f :=\!\! \sum_{(\b{x}_i, \b{x}_j) \in \mathcal{S}} (\b{x}_i - \b{x}_j)^\top \b{W} (\b{x}_i - \b{x}_j) + \!\!\sum_{(\b{x}_i, \b{x}_j) \in \mathcal{D}} (\b{x}_i - \b{x}_j)^\top \b{W}^{-1} (\b{x}_i - \b{x}_j) + \frac{1}{2} \|\b{W}\|_F^2, \\
& \text{s.t.}
& & \b{W} \in \mathcal{M} = \mathbb{S}_{++}^n,
\end{aligned}
\end{equation*}
where $\mathcal{S}$ and $\mathcal{D}$ denote the sets of similar and dissimilar points, respectively. We used class labels to randomly sample the points for these sets. This metric learning behaves like triplet loss in which the intra- and inter-class variances are decreased and increased, respectively, for better discrimination of classes \citep{ghojogh2020fisher}. 
The Euclidean gradient of this problem is $\mathbb{R}^{n \times n} \ni \nabla_E f(\b{W}) = \sum_{(\b{x}_i, \b{x}_j) \in \mathcal{S}} (\b{x}_i - \b{x}_j) (\b{x}_i - \b{x}_j)^\top - \sum_{(\b{x}_i, \b{x}_j) \in \mathcal{D}} (\b{W}^{-1} (\b{x}_i - \b{x}_j)) (\b{W}^{-1} (\b{x}_i - \b{x}_j))^\top  + \b{W}$.

\noindent
\textbf{$\bullet$ Results:}
We evaluated our mappings with geometric metric learning optimization on three public datasets, i.e., Fisher Iris, USPS digits, and MNIST digits. The average results over ten runs are reported in Table \ref{table_RiemMax_metric_learning}. In Iris data, both mappings have converged faster than RLBFGS, having a better converged cost function. In USPS and MNIST data, the Cholesky and ISR mappings have outperformed RLBFGS without mapping, respectively. 

\begin{table*}[!h]
\caption{Comparison of geometric metric learning by RLBFGS, with and without the proposed mappings, where exponential map is used in algorithms.}
\label{table_RiemMax_metric_learning}
\renewcommand{\arraystretch}{1.3}  
\centering
\scalebox{0.7}{    
\renewcommand{\arraystretch}{1.1}
\begin{tabular}{l | l | l | l| l | l }
\hline
\hline
Data & Algorithm & \#iters & conv. time & iter. time & last cost  \\
\hline
\multirow{ 3}{*}{Iris} 	
& 	  VTF (ISR) & 23.500$\pm$4.528 &	 2.461$\pm$1.174 &	 0.110$\pm$0.070 	& 1512.602$\pm$347.004 	  \\
& VTF (Chol.) &	25.000$\pm$5.676 & 	 2.474$\pm$1.009 	& 0.096$\pm$0.028 	& 1594.975$\pm$124.087   \\
& RLBFGS 	& 25.000$\pm$4.714 & 	 2.914$\pm$1.239 &	 0.113$\pm$0.037 	& 1620.207$\pm$125.989   \\
\hline    
\multirow{ 3}{*}{USPS} 	
& 	  VTF (ISR) & 14.700$\pm$3.234 &	 1.885$\pm$1.024 	& 0.133$\pm$0.094 &	 14223.215$\pm$0.001   \\
& VTF (Chol.) & 13.100$\pm$1.524 &	 1.246$\pm$0.217 	& 0.095$\pm$0.012 &	 14223.216$\pm$0.001 	 \\
& RLBFGS 	& 13.100$\pm$2.234 &	 1.307$\pm$0.454 	& 0.098$\pm$0.025 &	 14223.215$\pm$0.001 	 \\
\hline
\multirow{ 3}{*}{MNIST} 	
& 	  VTF (ISR) &  11.700$\pm$3.335 & 	 1.288$\pm$0.619 	& 0.108$\pm$0.041 	     & 6254.283$\pm$0.001 	  \\
& VTF (Chol.) &	13.100$\pm$3.479 &	 1.435$\pm$0.793 &	 0.104$\pm$0.029 	&      6254.284$\pm$0.001  \\
& RLBFGS 	& 12.100$\pm$3.381 &	 1.266$\pm$0.754 &	 0.099$\pm$0.024 & 	      6254.284$\pm$0.001   \\
\hline
\hline
\end{tabular}%
}
\end{table*}

\section{Conclusion and Future Direction}\label{section_conclusion}

In this paper, we proposed two mappings in the tangent space of SPD manifolds by inverse second root and Cholesky decomposition. The proposed mappings simplify the vector transports and adjoint vector transports to identity. These transports are widely used in RLBFGS quasi-Newton optimization, to identity. They also reduce the Riemannian metric to the Euclidean inner product which is more efficient computationally. Simulation results verified the effectiveness of the proposed mappings for two optimization tasks on SPD matrix manifolds. In this work, we focused on mappings for SPD manifolds which are widely used in machine learning and data science. A possible future direction is to extend the proposed mappings for other well-known Riemannian manifolds, such as Grassmann and Stiefel \citep{edelman1998geometry}, as well as other Riemannian optimization methods. This paper opens a new research path for such mappings in the tangent space and we conjecture that such mappings can make numerical Riemannian optimization more efficient.


\bibliography{acml21}








\clearpage
\begin{center}
\textbf{\large Vector Transport Free Riemannian LBFGS for Optimization on Symmetric Positive Definite Matrix Manifolds (Supplementary Material)}
\end{center}
\setcounter{equation}{0}
\setcounter{figure}{0}
\setcounter{table}{0}
\setcounter{section}{0}
\setcounter{page}{1}
\makeatletter


\section{Proof for Proposition 5}

\textbf{-- Upper bound analysis: } The most time consuming part of the RLBFGS algorithm is the recursive function GetDirection(.) defined in Section 2.4. In every call of recursion, Eqs. (1)-(3) are performed. The metric used in Eqs. (1) and (3) takes $\mathcal{O}(n^3)$ and $\mathcal{O}(n^2)$ before and after applying the mappings, respectively (cf. Table 1 in main paper). This is because matrix inversion takes $\mathcal{O}(n^3)$ time and the matrices are vectorized within trace operator in the algorithm implementation. The vector transport used in Eq. (2) takes $\mathcal{O}(n^3)$ and $\mathcal{O}(1)$ before and after applying the mappings, respectively (cf. Table 1 in main paper). Recursion in RLBGS is usually done for a constant $m$ number of times (see Ref. Ring and Wirth, 2020). Overall, the complexity of recursion is $\mathcal{O}(m n^3)$ and $\mathcal{O}(m n^2)$, for RLBFGS with and without mappings, respectively.  

\hfill\break
\noindent
\textbf{-- Lower bound analysis: }
Computing matrix inversion in the metric before applying our mappings (cf. Table 1 in main paper) takes $\Omega(n^2 \log n)$ (see article [1] referenced below) while metric takes $\Omega(n^2)$ after the mappings. The vector transport takes $\Omega(n^3)$ and $\Omega(1)$ before and after the mappings, respectively. Overall, the complexity of recursion is $\Omega(m n^3)$ and $\Omega(m n^2)$, for RLBFGS with and without mappings, respectively.  

\noindent
[1] A. Tveit, ``On the complexity of matrix inversion," Norwegian University of Science and Technology, Trondheim, Technical Report, 2003.

\hfill\break
\noindent
\textbf{-- Tight bound analysis: }
As the upper bound and lower bound complexities of the every algorithm are equal, we can conclude that the complexity bound is tight. Therefore, RLBFGS has time complexity $\Theta(m n^3)$ and $\Theta(m n^2)$, with and without our mappings, respectively. 
This shows that our proposed mappings improve the time of optimization. This time improvement shows off better if computation of gradients in Eq. (6) is not dominant in complexity. 

\section{Simulations for Larger Sample Size and Dimensionality}

Tables 2 and 3, in this Supplementary Material, report the average simulation results for large sample size, i.e., $N=100n^2$. 
In these tables, exponential map and Taylor approximation for retraction are used, respectively. 
We also have reported results for large dimensionality $d=100$ in these two tables.

\begin{table*}[!t]
\caption{Comparison of average results over ten runs where retraction with Taylor series expansion is used in algorithms and $K\in \{2,5\}$,  $n\in \{2,10\}$, $N=10n^2 \in \{40, 1000\}$. The $\#$iters, conv, iter, diff, and std are short for number of iterations, convergence, iteration, difference, and standard deviation, respectively.}
\label{table_RiemMax_10n2_K2_taylor}
\renewcommand{\arraystretch}{1.3}  
\centering
\scalebox{0.7}{    
\begin{tabular}{l | l | l | l | l | l | l | l | l}
\hline
\hline
$K$ & $n$ & Separation & Algorithm & \#iters & conv. time & time diff. std & iter. time & last cost \\
\hline
\multirow{18}{*}{2} & \multirow{9}{*}{2} & \multirow{ 3}{*}{Low} 	
& 		  VTF (ISR) 	&	 59.800$\pm$22.553 &	 93.135$\pm$88.437 & 19.280 	& 1.320$\pm$0.711 & 0.610$\pm$0.366 \\
& & 	& VTF (Chol.) 	&	 62.600$\pm$29.079 &	 106.545$\pm$117.633 	& 48.611 & 1.355$\pm$0.830 & 0.619$\pm$0.382 \\
& & 	& RLBFGS 	&	 	 59.800$\pm$26.803 &	 100.424$\pm$120.310 	& -- & 1.363$\pm$0.787 & 0.610$\pm$0.366 \\
\cline{3-9}                               
& & \multirow{ 3}{*}{Mid} 	
& 		  VTF (ISR) 	&	 45.800$\pm$16.982 &	 48.155$\pm$38.045 & 10.833 &	 0.900$\pm$0.454 & 0.482$\pm$0.497 \\
& & 	& VTF (Chol.) 	&	 47.000$\pm$18.074 &	 51.529$\pm$41.465 & 11.146 &	 0.930$\pm$0.483 & 0.482$\pm$0.497 \\
& & 	& RLBFGS 	&	 	 45.000$\pm$16.330 &	 48.150$\pm$37.285 & -- &	 0.921$\pm$0.458 & 0.482$\pm$0.497 \\
\cline{3-9}                            
& & \multirow{ 3}{*}{High} 	
& 		  VTF (ISR) 	&	 25.600$\pm$4.142 &	 9.816$\pm$3.837 	& 1.380 &    0.370$\pm$0.093 &  0.270$\pm$0.456   \\
& & 	& VTF (Chol.) 	&	 26.300$\pm$4.448 &	 10.540$\pm$4.190 & 1.932 &	 0.385$\pm$0.101 & 0.270$\pm$0.456   \\
& & 	& RLBFGS 	&	 	 25.500$\pm$4.353 &	 10.526$\pm$4.819 & -- &	 0.396$\pm$0.113 & 0.270$\pm$0.456   \\
\cline{2-9}   
& \multirow{9}{*}{10} & \multirow{ 3}{*}{Low} 	
& 		  VTF (ISR) 	&	 75.800$\pm$25.607 &	 166.736$\pm$137.600 & 63.108 &	 1.955$\pm$0.816 & 4.129$\pm$0.771\\
& & 	& VTF (Chol.) 	&	 74.500$\pm$22.629 &	 152.074$\pm$107.531 & 98.424 &	 1.855$\pm$0.682 & 4.129$\pm$0.771\\
& & 	& RLBFGS 	&	 	 74.800$\pm$25.442 &	 172.641$\pm$143.198 & -- & 	 2.054$\pm$0.834 & 4.129$\pm$0.771\\
\cline{3-9}  
& & \multirow{ 3}{*}{Mid}	
& 		  VTF (ISR) 	&	 50.900$\pm$13.956 &	 64.501$\pm$38.055 & 11.794 &	 1.170$\pm$0.395 &  3.472$\pm$1.154  \\
& & 	& VTF (Chol.) 	&	 52.100$\pm$13.892 &	 66.775$\pm$41.257 & 7.815 &	 1.184$\pm$0.410 &  3.472$\pm$1.154  \\
& & 	& RLBFGS 	&	 	 51.000$\pm$14.063 &	 70.132$\pm$46.421 &	-- & 1.264$\pm$0.450 &  3.472$\pm$1.154  \\
\cline{3-9}                                    
& & \multirow{ 3}{*}{High} 	
& 		  VTF (ISR) 	&	 43.600$\pm$5.522 &	 42.763$\pm$11.897 &	12.828 & 0.963$\pm$0.168  & 4.353$\pm$1.081   \\
& & 	& VTF (Chol.) 	&	 47.400$\pm$6.620 &	 50.438$\pm$15.187 & 8.731 &	 1.041$\pm$0.182  & 4.353$\pm$1.081   \\
& & 	& RLBFGS 	&	 	 44.300$\pm$6.464 &	 47.830$\pm$14.747 & -- &	 1.055$\pm$0.192  & 4.353$\pm$1.081   \\
\hline
\hline
\multirow{18}{*}{5} & \multirow{9}{*}{2} & \multirow{ 3}{*}{Low} 	
& 		  VTF (ISR) &	 262.500$\pm$126.532 &	 4430.577$\pm$4455.877 &	13196.851 & 13.849$\pm$6.991 & 0.082$\pm$0.281\\
& & 	& VTF (Chol.) &	 253.500$\pm$124.970 &	 4086.076$\pm$3967.586 &	10698.273 & 13.084$\pm$6.850 & 0.099$\pm$0.279\\
& & 	& RLBFGS &	 	 270.700$\pm$144.145 &	 5305.821$\pm$5495.383 &	-- & 15.560$\pm$8.452 & 0.090$\pm$0.282\\
\cline{3-9}   
& & \multirow{ 3}{*}{Mid} 	
& 		  VTF (ISR) &	 110.900$\pm$50.886 &	 731.318$\pm$769.639 	& 184.154 & 5.445$\pm$2.806    & 1.010$\pm$0.349\\
& & 	& VTF (Chol.) &	 112.200$\pm$60.736 &	 792.756$\pm$1053.832 & 416.142 &	 5.436$\pm$3.358 &  1.010$\pm$0.349\\
& & 	& RLBFGS &	 	 111.900$\pm$55.183 &	 814.994$\pm$975.737 	& -- & 5.827$\pm$3.289    & 1.010$\pm$0.349\\
\cline{3-9}   
& & \multirow{ 3}{*}{High} 	
& 		  VTF (ISR) &	 52.000$\pm$13.565 	& 116.485$\pm$69.535 & 12.185 &	 2.071$\pm$0.728 & 1.626$\pm$0.431  \\
& & 	& VTF (Chol.) &	 51.400$\pm$12.616 	& 114.032$\pm$67.916 & 19.838 &	 2.057$\pm$0.735 & 1.626$\pm$0.431  \\
& & 	& RLBFGS &	 	 53.400$\pm$14.081 	& 139.989$\pm$89.213 & -- &	 2.408$\pm$0.936 & 1.626$\pm$0.431  \\
\cline{2-9}
& \multirow{9}{*}{10} & \multirow{ 3}{*}{Low} 	
& 		  VTF (ISR) 	& 219.700$\pm$57.908 	& 2946.040$\pm$1513.372 	& 1239.375 & 12.579$\pm$3.507 & 6.643$\pm$0.662\\
& & 	& VTF (Chol.) 	& 226.100$\pm$86.010 	& 3272.111$\pm$2808.634 	& 2103.162 & 12.707$\pm$5.156 & 6.625$\pm$0.650\\
& & 	& RLBFGS 	& 	  231.300$\pm$64.484 	& 3607.215$\pm$1947.580 	& -- & 14.516$\pm$4.305 & 6.642$\pm$0.663\\
\cline{3-9}   
& & \multirow{ 3}{*}{Mid} 	
& 		  VTF (ISR) 	& 87.100$\pm$21.502 &	 413.616$\pm$223.510 & 54.437 &	 4.458$\pm$1.310 & 6.585$\pm$0.569\\
& & 	& VTF (Chol.) 	& 87.200$\pm$21.186 &	 405.484$\pm$214.985 & 47.791 &	 4.374$\pm$1.262 & 6.585$\pm$0.569\\
& & 	& RLBFGS 	& 	  87.400$\pm$20.403 &	 454.434$\pm$227.556 & -- &	 4.918$\pm$1.339 & 6.585$\pm$0.569\\
\cline{3-9}   
& & \multirow{ 3}{*}{High}  	
& 		  VTF (ISR) 	& 60.500$\pm$10.533 &	 181.113$\pm$76.070 & 16.691 &	 2.892$\pm$0.649 & 6.827$\pm$0.836\\
& & 	& VTF (Chol.) 	& 62.700$\pm$11.295 &	 197.565$\pm$93.653 & 24.531 &	 3.019$\pm$0.827 & 6.827$\pm$0.836\\
& & 	& RLBFGS 	& 	  58.900$\pm$11.040 &	 187.518$\pm$86.254 & -- &	 3.058$\pm$0.746 & 6.827$\pm$0.836\\
\hline
\hline
\end{tabular}%
}
\end{table*}

\begin{table*}[!t]
\caption{Comparison of average results over ten runs where exponential map is used in algorithms and $n\in \{2,10,100\}$, $N=100n^2 \in \{400, 10000, 1000000\}$, $K=2$.}
\label{table_RiemMax_100n2_K2_expm}
\renewcommand{\arraystretch}{1.3}  
\centering
\scalebox{0.7}{    
\begin{tabular}{l | l | l | l | l | l | l | l}
\hline
\hline
$n$ & Separation & Algorithm & \#iters & conv. time & time diff. std & iter. time & last cost \\
\hline
\multirow{9}{*}{2} & \multirow{ 3}{*}{Low} 
&	 VTF (ISR) &	 72.000$\pm$22.949 &	 127.902$\pm$98.372 	& 31.248 & 1.616$\pm$0.584 & 0.281$\pm$0.394\\
& &	 VTF (Chol.) &	 69.800$\pm$23.136 &	 116.250$\pm$86.811 	& 23.937 & 1.490$\pm$0.590 & 0.281$\pm$0.394\\
& &	 RLBFGS &	 	 68.900$\pm$24.875 &	 123.064$\pm$105.594 	& -- & 1.561$\pm$0.694 & 0.281$\pm$0.394    \\
\cline{2-8}                                                                         
& \multirow{ 3}{*}{Mid} 
&	 VTF (ISR) &	 59.400$\pm$16.460 &	 78.124$\pm$53.880 & 10.139 &	 1.210$\pm$0.421 & 0.520$\pm$0.392\\
& &	 VTF (Chol.) &	 54.900$\pm$12.862 &	 63.241$\pm$34.382 & 31.027 &	 1.082$\pm$0.332 & 0.516$\pm$0.393\\
& &	 RLBFGS &	 	 58.400$\pm$17.037 &	 80.042$\pm$62.183 & -- &	 1.240$\pm$0.497 & 0.520$\pm$0.392    \\
\cline{2-8}                                                                         
& \multirow{ 3}{*}{High} 
& 	     VTF (ISR) &	 23.300$\pm$2.627 &	 7.434$\pm$2.288 & 0.738 &	 0.313$\pm$0.057 & 0.102$\pm$0.288\\
& & 	 VTF (Chol.) &	 22.900$\pm$2.685 &	 7.273$\pm$2.084 & 0.873 &	 0.312$\pm$0.052 & 0.102$\pm$0.288\\
& & 	 RLBFGS &	 	 23.300$\pm$2.497 &	 7.917$\pm$2.135 & -- &	 0.335$\pm$0.053 & 0.102$\pm$0.288    \\
\hline
\hline
\multirow{9}{*}{10} & \multirow{ 3}{*}{Low} 
&	 VTF (ISR) 	&  63.400$\pm$16.728 	& 112.682$\pm$74.903 &	17.031 & 1.664$\pm$0.489 & 4.364$\pm$1.299\\
& &	 VTF (Chol.)&  65.000$\pm$17.994 	& 117.215$\pm$82.258 &	13.637 & 1.670$\pm$0.538 & 4.364$\pm$1.299\\
& &	 RLBFGS 	&  64.900$\pm$17.866 	& 126.987$\pm$87.668 &	-- & 1.819$\pm$0.561 & 4.364$\pm$1.299\\
\cline{2-8}  
& \multirow{ 3}{*}{Mid} 
&	 VTF (ISR) 	&  48.200$\pm$8.766 	& 58.096$\pm$23.089 & 10.899 &	 1.164$\pm$0.256 & 4.024$\pm$1.627\\
& &	 VTF (Chol.)&  49.800$\pm$11.811 	& 64.021$\pm$34.770 &	6.571 & 1.208$\pm$0.364 & 4.024$\pm$1.627\\
& &	 RLBFGS 	&  48.400$\pm$10.906 	& 64.099$\pm$32.746 &	-- & 1.251$\pm$0.361 & 4.024$\pm$1.627\\
\cline{2-8}  
& \multirow{ 3}{*}{High} 
& 	     VTF (ISR) 	&  45.100$\pm$7.385 &	 49.628$\pm$20.550 &	7.336 & 1.065$\pm$0.243 & 3.290$\pm$1.129\\
& & 	 VTF (Chol.)&  47.800$\pm$8.121 &	 55.217$\pm$22.622 & 6.572 &	 1.117$\pm$0.251 & 3.290$\pm$1.129\\
& & 	 RLBFGS 	&  45.200$\pm$7.131 &	 52.883$\pm$20.627 & -- &	 1.137$\pm$0.236 & 3.290$\pm$1.129\\
\hline
\hline
\multirow{9}{*}{100} & \multirow{ 3}{*}{Low} 
& 	     VTF (ISR) 	 &	 131.900$\pm$32.402 &	 54826.327$\pm$30454.295 & 5842.849 &	 389.761$\pm$121.060 & 63.703$\pm$2.843\\
& & 	 VTF (Chol.) &	 141.200$\pm$36.908 &	 59082.278$\pm$32385.087 & 3775.919 &	 392.658$\pm$110.916 & 63.703$\pm$2.843\\
& & 	 RLBFGS 	 &	 136.300$\pm$36.059 &	 56627.788$\pm$32602.224 & -- &	 387.367$\pm$119.404 & 63.703$\pm$2.843\\
\cline{2-8}                                    
& \multirow{ 3}{*}{Mid} 
& 	     VTF (ISR) 	 &	 87.800$\pm$20.225 	& 23635.623$\pm$11743.237 &	5738.406 & 259.026$\pm$52.061 & 61.263$\pm$4.734\\
& & 	 VTF (Chol.) &	 95.800$\pm$25.170 	& 28033.655$\pm$16196.321 &	1410.113 & 278.456$\pm$64.670 & 61.263$\pm$4.734\\
& & 	 RLBFGS 	 &	 94.900$\pm$25.653 	& 28527.524$\pm$17341.755 &	-- & 284.548$\pm$69.821 & 61.263$\pm$4.734\\
\cline{2-8}                                    
& \multirow{ 3}{*}{High} 
& 	     VTF (ISR) 	 &	 97.000$\pm$24.240 	& 29048.836$\pm$12857.812 &	3441.882 & 286.602$\pm$63.250 & 61.241$\pm$3.991\\
& & 	 VTF (Chol.) &	 105.100$\pm$27.477 	& 33967.543$\pm$15707.841 & 2014.574 &	 308.668$\pm$70.065 & 61.241$\pm$3.991\\
& & 	 RLBFGS 	 &	 103.200$\pm$26.389 	& 33101.986$\pm$15790.066 &	-- & 305.002$\pm$74.361 & 61.241$\pm$3.991\\
\hline
\hline
\end{tabular}%
}
\end{table*}

\begin{table*}[!h]
\caption{Comparison of average results over ten runs where retraction with Taylor series expansion is used in algorithms and $n\in \{2,10,100\}$, $N=100n^2 \in \{400, 10000, 1000000\}$, $K=2$.}
\label{table_RiemMax_100n2_K2_taylor}
\renewcommand{\arraystretch}{1.3}  
\centering
\scalebox{0.7}{    
\begin{tabular}{l | l | l | l | l | l | l | l}
\hline
\hline
$n$ & Separation & Algorithm & \#iters & conv. time & time diff. std & iter. time & last cost \\
\hline
\multirow{9}{*}{2} & \multirow{ 3}{*}{Low} 
&	 VTF (ISR)  &	 79.800$\pm$49.973 &	 206.935$\pm$343.153 &	105.106 & 1.839$\pm$1.340 & 0.403$\pm$0.578   \\
& &	 VTF (Chol.)&  	 72.200$\pm$20.004 &	 125.636$\pm$83.113  & 334.239 &    1.608$\pm$0.536 & 0.402$\pm$0.576   \\
& &	 RLBFGS  	& 	 83.900$\pm$51.054 &	 237.936$\pm$364.382 &	-- & 2.064$\pm$1.412 & 0.404$\pm$0.578   \\
\cline{2-8}                                    
&\multirow{ 3}{*}{Mid} 
&	 VTF (ISR)  &	 48.100$\pm$12.688 &	 49.277$\pm$30.486 &	12.185 & 0.946$\pm$0.334 & 0.196$\pm$0.436   \\
& &	 VTF (Chol.)&  	 50.100$\pm$14.271 &	 56.329$\pm$40.420 &	14.674 & 1.017$\pm$0.422 & 0.196$\pm$0.436   \\
& &	 RLBFGS  	& 	 51.000$\pm$13.622 &	 61.036$\pm$39.550 &	-- & 1.099$\pm$0.412 & 0.196$\pm$0.436   \\
\cline{2-8}                                  
&\multirow{ 3}{*}{High} 
& 	 VTF (ISR)  &	     25.300$\pm$3.945 &	 9.887$\pm$3.903 & 3.565  	 &   0.377$\pm$0.101 & 0.111$\pm$0.261  \\
& & 	 VTF (Chol.)&  	 26.500$\pm$4.927 &	 11.003$\pm$5.644  & 4.938 &	 0.395$\pm$0.124 & 0.111$\pm$0.261      \\
& & 	 RLBFGS  	& 	 25.100$\pm$4.012 &	 10.203$\pm$4.224  & -- &	 0.392$\pm$0.106 & 0.111$\pm$0.261      \\
\hline
\hline
\multirow{9}{*}{10} & \multirow{ 3}{*}{Low} 
&	 VTF (ISR) &	 65.500$\pm$18.435 &	 123.552$\pm$77.898 &	14.558 & 1.747$\pm$0.560 & 4.164$\pm$1.142\\
& &	 VTF (Chol.) &	 66.500$\pm$19.558 &	 122.281$\pm$79.423 &	19.994 & 1.688$\pm$0.569 & 4.164$\pm$1.142\\
& &	 RLBFGS &	 	 65.200$\pm$18.743 &	 130.098$\pm$84.241 &	-- & 1.835$\pm$0.622 & 4.164$\pm$1.142\\
\cline{2-8}                                                                
& \multirow{ 3}{*}{Mid} 
&	 VTF (ISR) &	 40.600$\pm$5.211 	& 38.915$\pm$12.591 &	8.862 & 0.938$\pm$0.177 & 4.589$\pm$1.278\\
& &	 VTF (Chol.) &	 41.000$\pm$5.538 	& 38.732$\pm$12.699 &	9.319 & 0.924$\pm$0.171 & 4.589$\pm$1.278\\
& &	 RLBFGS &	 	 40.400$\pm$6.186 	& 41.513$\pm$15.303 &	-- & 0.998$\pm$0.213 & 4.589$\pm$1.278\\
\cline{2-8}                                                                
& \multirow{ 3}{*}{High} 
& 	     VTF (ISR) &	 44.400$\pm$6.310 &	 46.957$\pm$15.325 & 11.008 &	 1.032$\pm$0.198 & 3.786$\pm$1.113\\
& & 	 VTF (Chol.) &	 46.000$\pm$6.616 &	 50.393$\pm$15.950 & 12.688 &	 1.070$\pm$0.195 & 3.786$\pm$1.113\\
& & 	 RLBFGS &	 	 44.200$\pm$7.099 &	 50.707$\pm$18.070 & -- &	 1.116$\pm$0.222 & 3.786$\pm$1.113\\
\hline
\hline
\multirow{9}{*}{100} & \multirow{ 3}{*}{Low} 	
&     VTF (ISR) 	& 118.500$\pm$11.404 &	 39986.299$\pm$6719.992 & 2227.745 &	 335.341$\pm$26.657  &  62.241$\pm$5.708\\
& 	& VTF (Chol.) 	& 124.400$\pm$11.157 &	 42882.658$\pm$7219.289 & 3669.708 &	 342.538$\pm$28.830  &  62.241$\pm$5.708\\
& 	& RLBFGS 	 	& 121.300$\pm$11.235 &	 41171.059$\pm$7987.985 & -- &	 336.547$\pm$36.783  &  62.241$\pm$5.708\\
\cline{2-8}  
&\multirow{ 3}{*}{Mid} 	
&     VTF (ISR) 	& 103.800$\pm$24.679 &	 32533.037$\pm$16124.089 & 3314.944 &	 297.268$\pm$76.267 & 60.256$\pm$3.662\\
& 	& VTF (Chol.) 	& 111.900$\pm$28.781 &	 37328.925$\pm$18924.066 & 2916.610 &	 314.512$\pm$82.724 & 60.256$\pm$3.662\\
& 	& RLBFGS 	 	& 111.300$\pm$27.941 &	 36926.995$\pm$17705.262 & -- &	 314.564$\pm$76.642 & 60.256$\pm$3.662\\
\cline{2-8}  
&\multirow{ 3}{*}{High} 	
&     VTF (ISR) 	& 96.300$\pm$25.880 	& 28818.656$\pm$14923.845 & 1492.788 &	 283.824$\pm$67.267 & 60.601$\pm$4.436\\ 
& 	& VTF (Chol.) 	& 103.900$\pm$26.126 	& 33669.987$\pm$16494.757 & 3368.579 &	 308.353$\pm$73.281 & 60.601$\pm$4.436\\
& 	& RLBFGS 	 	& 102.800$\pm$25.258 	& 31630.478$\pm$15755.934 & -- &	 292.771$\pm$68.396 & 60.601$\pm$4.436\\
\hline
\hline
\end{tabular}%
}
\end{table*}

\section{Cost Difference Progress for $N=10n^2$ Sample Size Simulations}

The log-scale cost difference for simulations of $N=10n^2$ are depicted in Figs. 1 and 2 of this Supplementary Material, for $K=2$ and $K=5$, respectively. 

\begin{figure}[!h]
\centering
\includegraphics[width=\textwidth]{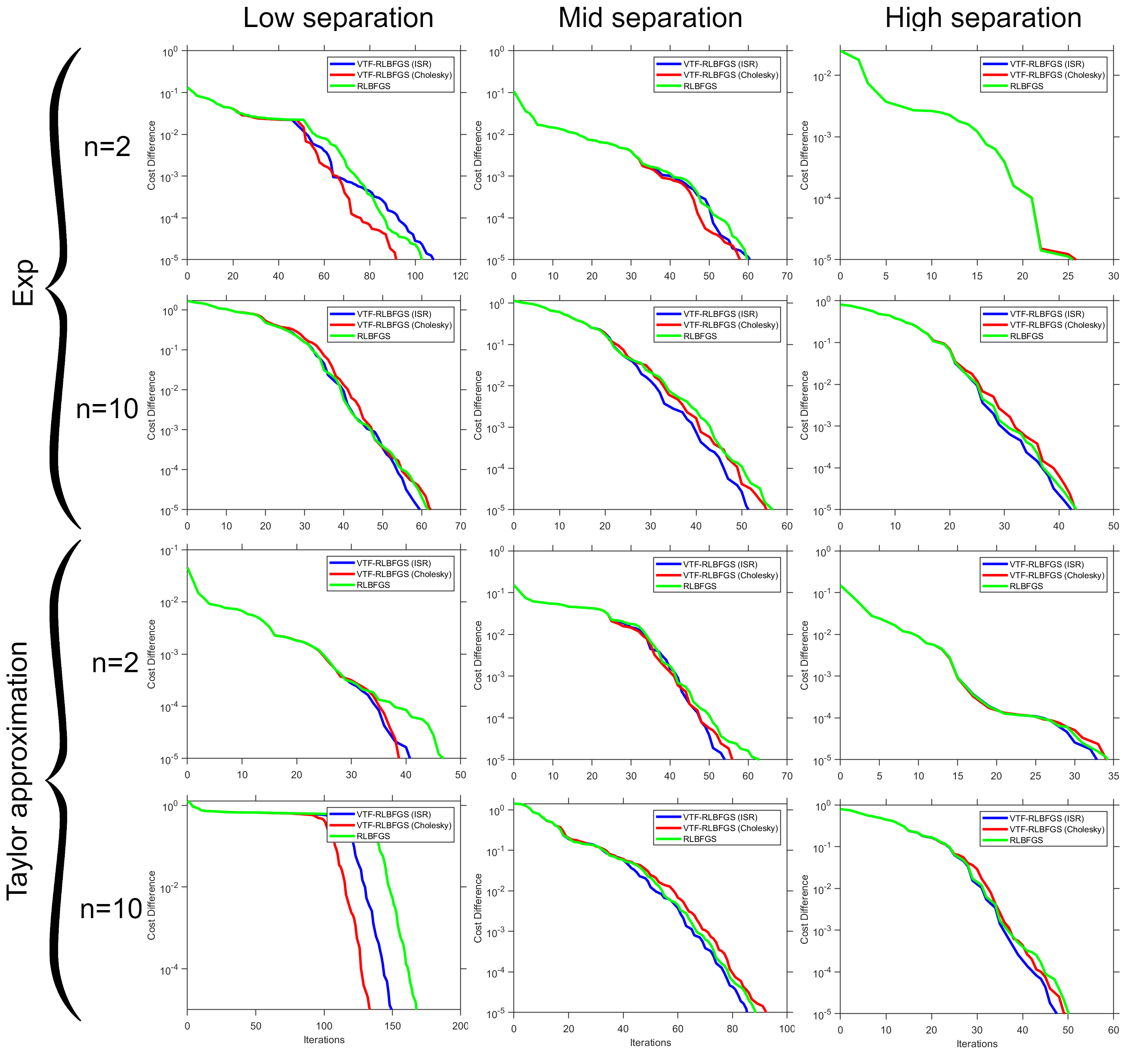}
\caption{Comparison of the proposed VTF-RLBFGS algorithm (using ISR and Cholesky) with RLBFGS in their cost differences. In these experiments, we had $K=2$.}
\label{figure_cost_10n2_K2_exp_taylor}
\end{figure}

\begin{figure}[!h]
\centering
\includegraphics[width=\textwidth]{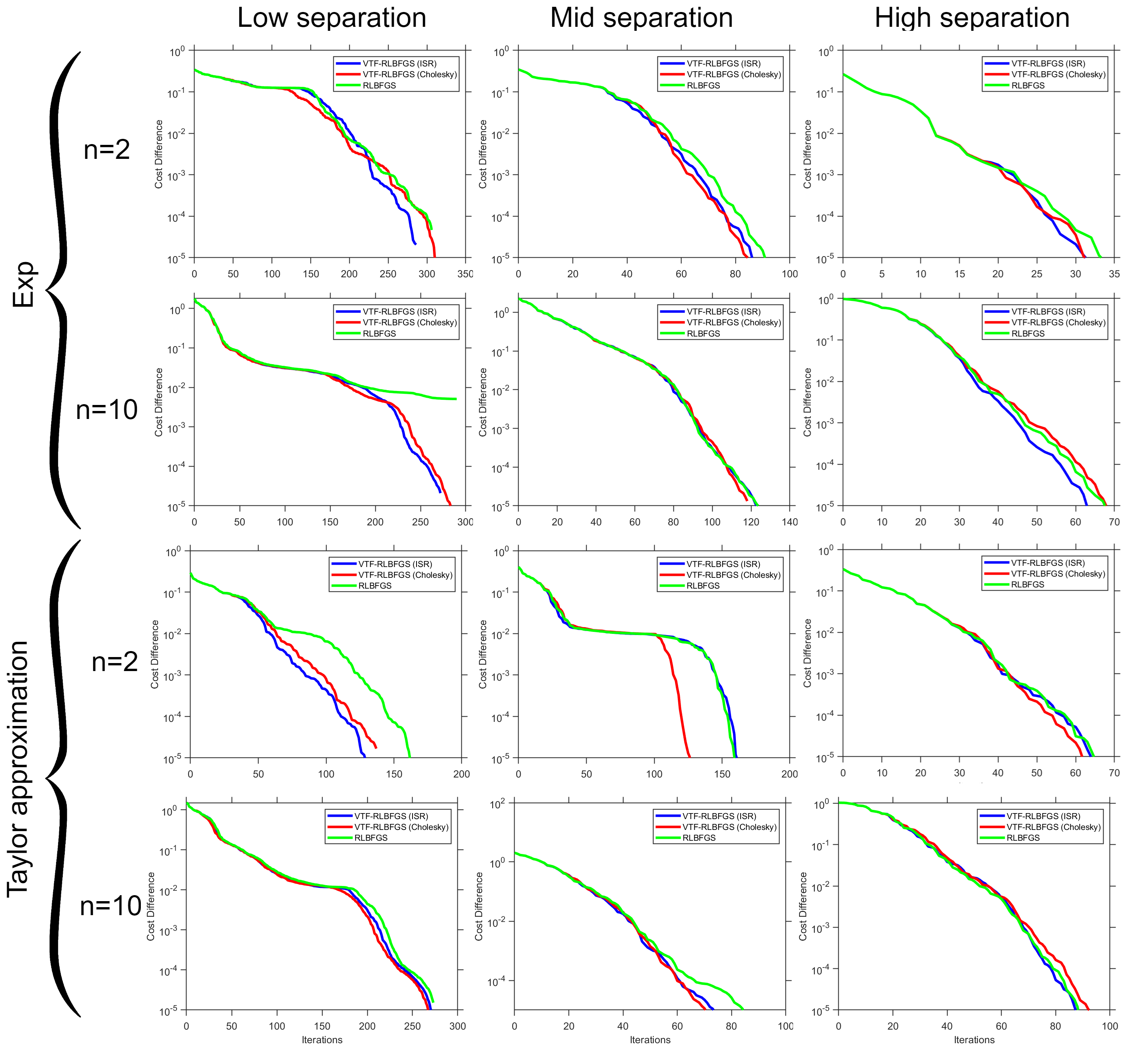}
\caption{Comparison of the proposed VTF-RLBFGS algorithm (using ISR and Cholesky) with RLBFGS in their cost differences. In these experiments, we had $K=5$.}
\label{figure_cost_10n2_K5_exp_taylor}
\end{figure}

\section{Cost Difference Progress for $N=100n^2$ Sample Size Simulations}

The log-scale cost difference for simulations of $N=100n^2$ are depicted in Figs. 3 and 4 of this Supplementary Material, where exponential map and Taylor approximation for retraction are used, respectively. 

\begin{figure}[!h]
\centering
\includegraphics[width=\textwidth]{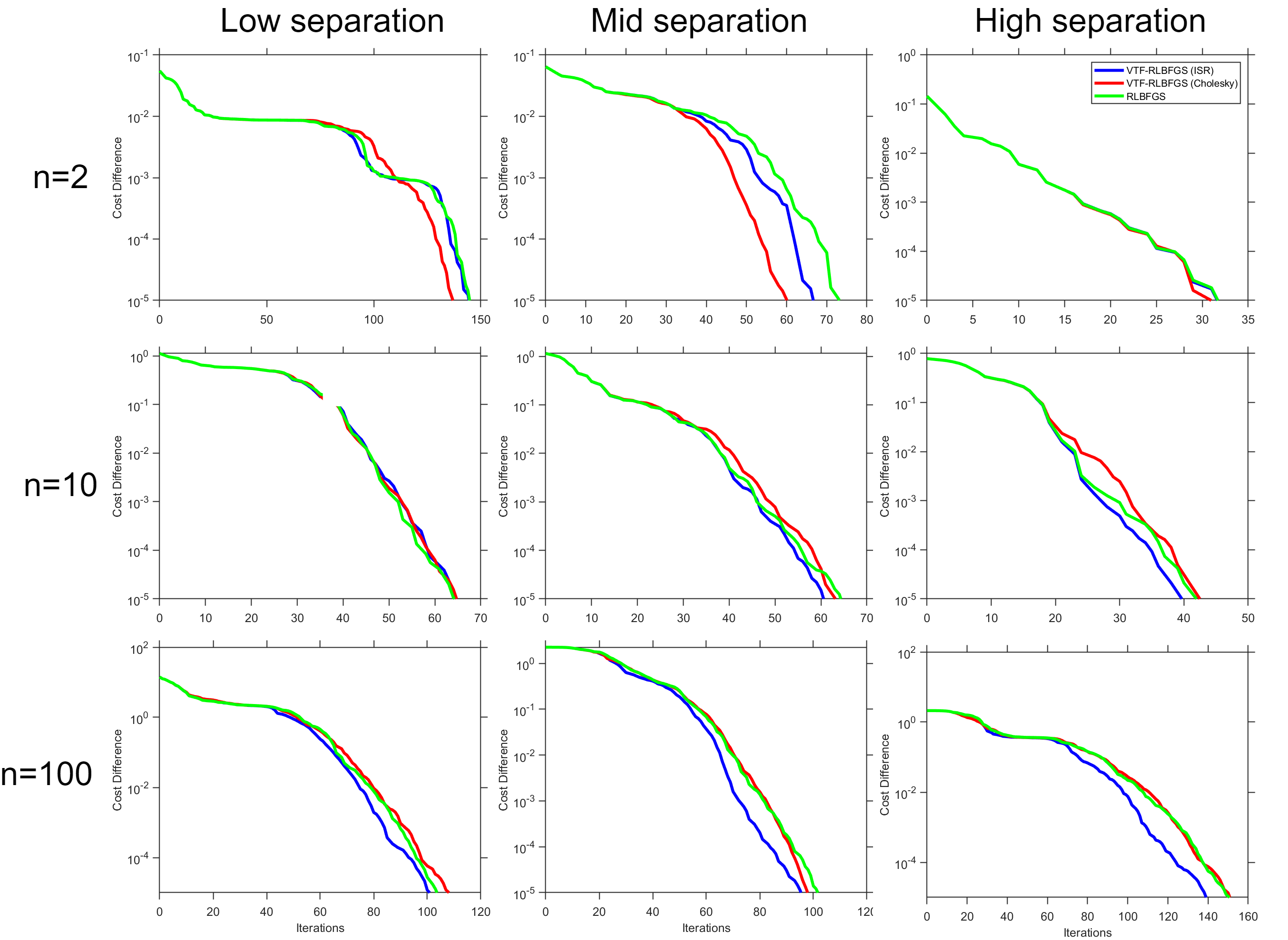}
\caption{Comparison of the proposed VTF-RLBFGS algorithm (using ISR and Cholesky) with RLBFGS in their cost differences. In these experiments, exponential map is used in optimization procedure and we had $K=2$.}
\label{figure_cost_100n2_K2_exp}
\end{figure}

\begin{figure}[!h]
\centering
\includegraphics[width=\textwidth]{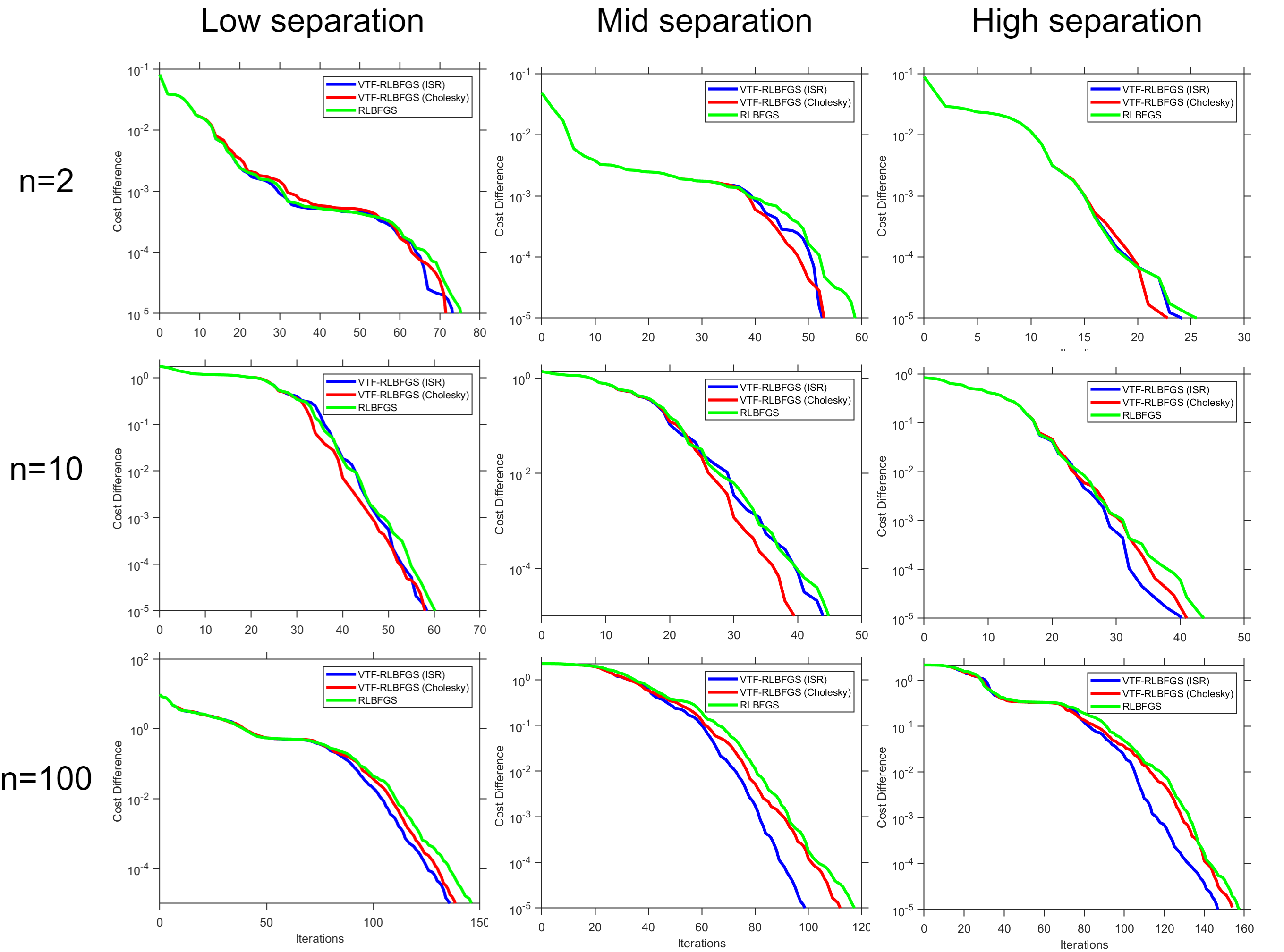}
\caption{Comparison of the proposed VTF-RLBFGS algorithm (using ISR and Cholesky) with RLBFGS in their cost differences. In these experiments, Taylor series expansion is used for approximation of retraction operator and we had $K=2$.}
\label{figure_cost_100n2_K2_taylor}
\end{figure}

\end{document}